\documentclass[10pt]{article}
\usepackage[T1]{fontenc}
\usepackage{amsmath,amsfonts,amsthm,mathrsfs,amssymb,cite}
\usepackage{graphicx}
\usepackage{tabularx}
\usepackage{booktabs}
\usepackage{subfigure}
\usepackage{placeins}
\usepackage{color}
\usepackage{xcolor}
\usepackage{indentfirst}

\numberwithin{equation}{section}
\newcommand{\R}{{\mathbb R}}
\newcommand{\Z}{{\mathbb Z}}

\newcommand{\C}{{\mathbb C}}

\newcommand{\be}{\begin{eqnarray}}
\newcommand{\ben}{\begin{eqnarray*}}
\newcommand{\en}{\end{eqnarray}}
\newcommand{\enn}{\end{eqnarray*}}

\newcommand{\curl}{{\rm curl\,}}

\newcommand{\grad}{{\rm grad\,}}
\newcommand{\divv}{{\rm div\,}}

\newcommand{\G}{\Gamma}

\newtheorem{theorem}{Theorem}[section]
\newtheorem{lemma}[theorem]{Lemma}
\newtheorem{corollary}[theorem]{Corollary}
\newtheorem{definition}[theorem]{Definition}
\newtheorem{remark}[theorem]{Remark}

\definecolor{rot}{rgb}{0.000,0.000,0.000}
\definecolor{rot1}{rgb}{0.000,0.000,0.000}
\definecolor{rot2}{rgb}{0.000,0.000,0.000}

\begin{document}
\renewcommand{\theequation}{\arabic{section}.\arabic{equation}}
\begin{titlepage}
\title{\bf Inverse source problems in elastodynamics}
\author{
Gang Bao\ \ ({\sf baog@zju.edu.cn}) \ \\
 {\small School of Mathematical Sciences, Zhejiang University, Hangzhou 310027, China}\\ \\
Guanghui Hu\ \ ({\sf hu@csrc.ac.cn}) \ \\
 {\small Beijing Computational Science Research Center, Beijing 100094, China}\\ \\
 Yavar Kian\ \ ({\sf yavar.kian@univ-amu.fr}) \ \\
 {\small Aix Marseille Univ, Universit\'e de Toulon, CNRS, CPT, Marseille, France}\\ \\
Tao Yin\ \ ({\sf taoyinzju@gmail.com}) \ \\
 {\small School of Mathematical Sciences, Zhejiang University, Hangzhou 310027, China}\\
 {\small Laboratoire Jean Kuntzmann, Universit\'{e} Grenoble-Alpes, 700 Avenue Centrale,
 38401 Saint-Martin-d'H\`{e}res, France}
}
\date{}
\end{titlepage}
\maketitle
\vspace{.2in}
\begin{abstract}
We are concerned with time-dependent inverse source problems in elastodynamics. The source term is supposed to be the product of a spatial function and a temporal function with compact support. We present frequency-domain and time-domain approaches to show uniqueness in determining the spatial function from wave fields on a large sphere over a finite interval. Stability estimate of the temporal function from the data of one receiver and uniqueness result using partial boundary data are proved. Our arguments rely heavily on the use of the Fourier transform, which motivated inversion schemes that can be easily implemented. A Landweber iterative algorithm for recovering the spatial function and a non-iterative inversion scheme based on the uniqueness proof for recovering the temporal function are proposed. Numerical examples are demonstrated in both two and three dimensions.

\vspace{.2in} {\bf Keywords:} Inverse source problems, Lam\'e system, uniqueness, Landweber iteration, Fourier transform.
\end{abstract}
\section{Introduction}
Consider the radiation of elastic (seismic) waves from a time-varying source term $F(x,t)$, $x\in \R^3$,
embedded in an infinite and homogeneous elastic medium. The real-valued radiated field is governed by the inhomogeneous Lam\'e system:
\be\label{lame}
\rho\partial_{tt}U(x,t)=\nabla\cdot \sigma(x,t)+F(x,t), \quad x=(x_1,x_2,x_3)\in \R^3,\; t>0
\en
together with the initial conditions
\be\label{eq:12}
U(x,0)=\partial_tU(x,0)=0,\quad x\in \R^3.
\en
Here, $\rho>0$ denotes the density,  $U=(u_1,u_2,u_3)^\top$ is the displacement vector, $\sigma=\sigma(U)$ is
the stress tensor and $F$ is the source term which causes the elastic vibration in $\R^3$.
 By Hooke's law,  the stress tensor relates to the stiffness tensor $\tilde{C}=(C_{ijkl})_{i,j,k,l=1}^3$
 via the identity $\sigma(U):=\tilde{C}:\nabla U$,
where the action of $\tilde{C}$ on a matrix $A=(a_{ij})_{i,j=1}^3$ is defined as
\ben
\tilde{C}:A=(\tilde{C}:A)_{ij}=\displaystyle\sum_{k,l=1}^3 C_{ijkl}\; a_{kl}.
\enn
In an isotropic and homogeneous elastic medium, the stiffness tensor is characterized by
\be\label{C}
 C_{ijkl}(x)=\lambda \delta_{i,j}\delta_{k,l}+\mu ( \delta_{i,k}\delta_{j,l}+ \delta_{i,l}\delta_{j,k}).
\en where the Lam\'e constants satisfy $\mu>0, 3\lambda+2 \mu>0$,
Hence, the Lam\'e system \eqref{lame} can be rewritten as
\ben
&&\rho\,\partial_{tt}\,U(x,t)= \mathcal{L}_{\lambda,\mu} U(x,t)+F(x,t),\quad (x,t)\in \R^3\times \R^+,\\
&&\mathcal{L}_{\lambda,\mu}U:=-\mu \nabla\times\nabla\times U+
(\lambda+2\mu)\nabla\nabla\cdot U
=\mu \Delta U+(\lambda+\mu)\nabla \nabla\cdot U .
\enn
Note that the above equation has a more complex form than the scalar wave equation, because it  accounts for both longitudinal and transverse motions.
Throughout this paper it is supposed that $\rho, \lambda,\mu$ are given as a prior data and that the dependence of the source term on
time and space variables are separated, that is,
\be\label{eq:13}
F(x,t)=f(x)\;g(t).
\en
In other words,
  the source term is a product of the spatial function $f$ and the temporal function $g$.
 Moreover, we suppose that $f$ is compactly supported in the space region $B_{R_0}:=\{x: |x|<R_0\}$ and
 the source radiates only over a finite time period $[0,T_0]$ for some $T_0>0$.
 This implies that \textcolor{rot}{$g(t)= 0$} for $t\geq T_0$ and $t\leq 0$.
 The source term (\ref{eq:13}) can be regarded as an approximation of the elastic pulse and are commonly used in modeling vibration phenomena in seismology.

 \textcolor{rot2}{Inverse hyperbolic problems have attacted considerable
 attention over the last years. Most of the eixsting works treated scalar acoustic wave equations. We refer to
 Bukhgaim \& Klibanov \cite{BK1981}, Klibanov \cite{Klibanov1992},  Yamamoto \cite{Yamamoto,Ya99},
 Kha\v{\i}darov \cite{Kh}, Isakov \cite{Isakov93, Isakov}, Imanuvilov \& Yamamoto \cite{OY01,OY2001}, Choulli and Yamamoto
 \cite{CY2006}}, \textcolor{rot}{Kian, Sambou \& Soccorsi \cite{KSS16}} \textcolor{rot}{ and
 the recent work by Jiang, Liu \& Yamamoto \cite{JLY}
 for uniqueness and stability of inverse source problems using Carleman estimates},  \textcolor{rot}{and refer also to
 Fujishiro \& Kian \cite{FK} for results of recovery of a time-dependent source}.
 \textcolor{rot}{ In addition, it is worth to mention the work of Rakesh \&  Symes \cite{Rakesh1988}, dealing with coefficient
 determination problems based on the construction of appropriate geometric optic solutions.
 There are also rich references on inverse problems arising in the context of linear elasticity.
 Many investigations are devoted to mathematical and numerical techniques for the identification of
 elastic coefficients and buried objects of a geometrical nature
 (such as cracks, cavities and inclusions) in the time-harmonic regime; see e.g.,
 the review article \cite{Rew2005} by Bonnet \& Constantinescu , the monograph  \cite{Ammari} by Ammari et.al. and references therein.
 Due to our limited knowledge, we have found only
a few mathematical works on inverse source problems for the time-dependent Lam\'e system.}
 In \cite{Ammari13}, a time-reversal imaging algorithm
based on a weighted Helmholtz decomposition was proposed for reconstructing $f$ in a homogeneous isotropic medium,
where the temporal function takes the special form $g(t)= d\delta(t)/dt$.

  \textcolor{rot2}{This paper concerns uniqueness and numerical reconstructions of  $f$ or $g$  from radiated elastodynamic fields over a finite time interval.
 Such kind of inverse problems have many significant applications in biomedical engineering (see, e.g.,\cite{Ammari}) and geophysics
 (see, e.g., \cite{AR}). The Uniqueness issue is important in inverse scattering theory,
 while it provides insight into whether the measurement data are sufficient for recovering the unknowns and ensures uniqueness of global minimizers in iterative schemes.
  Being different from previously mentioned existing works,
 our uniqueness proofs rely heavily on the use of the
 Fourier transform, which motivated novel inversion schemes that can be easily implemented. Our arguments carry over the scalar wave equations without any additional difficulties.
 We believe that the Fourier-transform-based approach explored in this paper  would also lead to
 stability estimates of our inverse problems, which deserves to be further investigated in future.
We shall address the following inverse issues:
\begin{description}
 \item[(i)]  Uniqueness in recovering $f$ from emitted
waves on a closed surface surrounding the source. We present frequency-domain and time-domain approaches for recovering the spatial function.
The frequency-domain approach is of independent interest, while it reduces the time-dependent inverse problem  to an inverse scattering
problem in the Fourier domain with multi-frequency data.
Our arguments are motivated by recent studies on inverse source problems for the time-harmonic Helmholtz equation with
multi-frequency data (see e.g., \cite{BLT,BLRX,AMR,EV,BLLT}).
The time-domain approach is inspired by the Lipschitz stability estimate of source terms
for the scalar acoustic wave equation
with additional a prior  assumptions;  see, e.g., \cite{Ya99,Yamamoto, OY01, OY2001}. A Landweber iterative algorithm is proposed
for recovering $f$ in 2D and numerical tests are presented to show validity and effectiveness of
the proposed inversion scheme; see Section \ref{test:spatial}.
\item[(ii)] Stability estimate of $g$ from measured data of one receiver.
Under the assumption that the spatial function  does not vanish at the position of the receiver, we
estimate a vector-valued temporal functions in
Section \ref{sec:temporal}. The stability estimate relies on an explicit expression of the solution in terms of $f$
and $g$. Such an idea seems well-known in the case of scalar acoustic wave equations, but to the best of our knowledge not
available for time-dependent Lam\'e systems.
\item[(iii)] Unique determination of $g$ from partial boundary measurement data.
If the spatial function $f$ is known to be not a non-radiating source (see Definition \ref{def}), we prove that
the temporal function $g$ can be uniquely determined by the time-domain data on any
subboundary of a large sphere; see Theorem \ref{TH:temporal}. The uniqueness proof is based on the Fourier
transform and yields a non-iterative inversion scheme in subsection \ref{subsect:temporal}. Numerical examples are demonstrated to verify our theory.
\end{description}
The remaining part is organized as following. In Section \ref{Pre},
preliminary studies of the time-dependent Lam\'e system are carried out.
Unique determination of spatial and temporal functions will be presented
in Sections \ref{sec:spatial} and \ref{sec:temporal}, respectively.
In particular, as a bi-product of the Fourier-domain approach presented in subsection \ref{sec:Fouier}, we show uniqueness
in recovering a source term of the time-dependent Schr\"odinger equation.
Finally,
Numerical tests are reported in Section \ref{sec:numerics} and
proofs of several lemmas are postponed to the appendix in Section \ref{Appendix}. }

\section{Preliminaries}\label{Pre}

\textcolor{rot}{For all $r>0$, we denote by $B_r$ the open ball of $\R^3$ defined by $B_r:=\{x\in \R^3:\ |x|<r\}$.}
By Helmholtz decomposition, \textcolor{rot}{the function $f\in (L^2(\R^3))^3$ supported in $B_{R_0}$}  admits a unique decomposition of the form (see Lemma \ref{HD-uniqueness} in the Appendix)
\be\label{f}
f(x)=\nabla f_p(x)+\nabla\times f_s(x),\quad \nabla\cdot f_s\equiv 0,
\en
where $f_p\in H^1(B_{R_0})$, $f_s\in H_{\curl}(B_{R_0}):=\{u: u\in (L^2 (B_{R_0}))^3, \curl u\in (L^2 (B_{R_0}))^3\}$ also have compact support in $B_{R_0}$. \textcolor{rot}{We choose also $g\in\mathcal C(\R)$ supported in $[0,T_0]$.}
 By the completeness theorem (see \cite[Theorem 3.3]{A73} or \cite[Chapter 4.1.1]{AR}), there exist vector-valued functions $U_p(x,t)$ and $U_s(x,t)$ such that $U(x,t)$ can be expressed as
\be\label{Ud}
U=U_p+U_s,\quad U_p=\nabla\,u_p,\quad U_s=\nabla\times\,u_s,\quad \nabla\cdot\,u_s=0.
\en
Moreover, the scalar function $u_p$ and the vector function $u_s$ satisfy the inhomogeneous wave equations
\be\label{eq:3}
\frac{1}{c_\alpha^2}\,\partial_{tt}\,u_{\alpha}-\Delta u_\alpha=\frac{1}{\gamma_\alpha}f_\alpha(x)g(t)\quad\mbox{in}\quad\R^3\times(0,+\infty),\qquad \alpha=p,s,
\en
together with the initial conditions
\ben
u_\alpha|_{t=0}=\partial_t u_{\alpha}|_{t=0}=0\quad\mbox{in}\quad\R^3.
\enn
Note that
\be\label{2.3}
c_p:=\sqrt{(\lambda+2\mu)/\rho}, \quad c_s:=\sqrt{\mu/\rho},\quad\;\gamma_p:=\lambda+2\mu,\;\gamma_s:=\mu,
\en
\textcolor{rot2}{and that $\lambda+2\mu>0$ since $\mu>0$, $3\lambda+2\mu>0$.}
This implies that $U_p$ and $U_s$ propagate at different wave speeds, which will be referred to as compressional waves (or simply P-waves) and shear waves (or simply S-waves), respectively.

 It is well-known that the
electrodynamic Green's tensor $G(x,t)=(G_{ij}(x,t))_{i,j=1}^3\in \C^{3\times 3}$, which satisfies
\ben
&&\rho \partial_{tt} G(x,t)e_j-\nabla\cdot \sigma(x,t)=-\delta(x)\delta(t)e_j,\quad j=1,2,3,\\
&& G(x,0)=\partial_t G(x,0)=0,\quad x\neq y,
\enn
is given by (see e.g., \cite{Costabel})
\be\nonumber
&&G_{i,j}(x,t)=\frac{1}{4\pi\rho|x|^3}\left\{ t^2\left( \frac{x_jx_k}{|x|^2}\delta(t-|x|/c_p)+(\delta_{jk}-\frac{x_jx_k}{|x|^2})\delta(t-|x|/c_s)
 \right)\right\} \\ \label{G}
&&\qquad\qquad+\frac{1}{4\pi\rho|x|^3}\left\{ t \left(3\frac{x_jx_k}{|x|^2}-\delta_{jk}\right)\left(\Theta(t-|x|/c_p)-\Theta(t-|x|/c_s)   \right)\right\}.
\en
Here, $\delta_{ij}$ is the Kronecker symbol, $\delta$ is the Dirac distribution, $\Theta$ is the Heaviside function and
$e_j$ ($j=1,2,3$) are the unit vectors in $\R^3$.
Physically, the Green's tensor $G(x,t)$ is the response of the Lam\'e system to a point body force at the origin that emits an impulse at time $t=0$.
Using the above Green's tensor, the solution $U$ to the inhomogeneous Lam\'e system \eqref{lame}  can be represented as
\be\label{U}
U(x,t)=\int_{0}^\infty\int_{\R^3} G(x-y, t-s) f(y)g(s)\,dx ds,\quad x\in \R^3,\;\textcolor{rot}{t\in\mathbb R}.
\en
\textcolor{rot}{Note that, since supp$(g)\subset[0,+\infty)$, for every $t\in(-\infty,0]$ and $x\in \mathbb R^3$, we  have $U(x,t)=0$.}
Throughout the paper we define
\be\label{Tps}
T_p:=T_0+(R+R_0)/c_p,\quad T_s:=T_0+(R+R_0)/c_s,
\en
for some $R>R_0$.
Obviously, it holds that $T_s>T_p$, since $c_p>c_s$ by (\ref{2.3}).
The following lemma states that the wave fields over $B_R$ must vanish after a finite time that depends on $R$ and the support of  $f$ and $g$.
\begin{lemma}\label{lem:1}
We have $U(x,t)\equiv 0$ for all $x\in B_R$ and  $t>T_s$.
\end{lemma}
\begin{proof} For $x=(x_1,x_2,x_3)^\top$, $y=(y_1,y_2,y_3)^\top\in \R^3$, write $x\otimes y= x y^\top\in \R^{3\times 3}$ and $\hat{x}=x/|x|$ for simplicity. Introduce
\ben
V(x,t)=\Theta(t-|x|/c_p)-\Theta(t-|x|/c_s).
\enn
Combining (\ref{U}) and (\ref{G}), we have
\ben
U(x,t)&=&\int_{0}^\infty\int_{\R^3} \frac{(t-s)^2 (\hat{x}-\hat{y})\otimes(\hat{x}-\hat{y})}{4\pi\rho|x-y|^3} \; \delta(t-s-\frac{|x-y|}{c_p}) f(y)g(s)\,dy ds\\
&&+\int_{0}^\infty\int_{\R^3} \frac{(t-s)^2 [\textbf{I}- (\hat{x}-\hat{y})\otimes(\hat{x}-\hat{y})]}{4\pi\rho|x-y|^3} \delta(t-s-\frac{|x-y|}{c_s}) f(y)g(s)\,dy ds\\
&&+\int_{0}^\infty\int_{\R^3} \frac{(t-s) [3(\hat{x}-\hat{y})\otimes(\hat{x}-\hat{y}) -\textbf{I }] }{4\pi\rho|x-y|^3} V(x-y,t-s)\, f(y)g(s)\,dyds\\
&=&\int_{|y-x|\leq c_p(\textcolor{rot2}{t+T_0})} \frac{ (\hat{x}-\hat{y})\otimes(\hat{x}-\hat{y})}{4\pi(\lambda+2\mu)|x-y|} \; g(t-\frac{|x-y|}{c_p}) f(y)\,dy ds\\
&&+\int_{|y-x|<c_s(\textcolor{rot2}{t+T_0})} \frac{ \textbf{I}- (\hat{x}-\hat{y})\otimes(\hat{x}-\hat{y})}{4\pi\mu|x-y|} g(t-\frac{|x-y|}{c_s}) f(y)\,dy ds\\
&&+\int_{0}^{T_0}\int_{B_R} \frac{(t-s) [3(\hat{x}-\hat{y})\otimes(\hat{x}-\hat{y}) -\textbf{I }] }{4\pi\rho|x-y|^3} \,V(x-y,t-s)\, f(y)g(s)\,dyds,
\enn
where  $\textbf{I}$ denotes the 3-by-3 unit matrix.
For $t>T_0+(R+R_0)/c_s$, one can readily observe that
\ben
g(t-\frac{|x-y|}{c_s})=g(t-\frac{|x-y|}{c_p})=0,\quad V(x-y,t-s)=0
\enn
uniformly in all $x\in B_R, y\in B_{R_0}$ and $s\in (0,T_0)$, which implies the desired result.
\end{proof}

Denote by $\hat{f}$ the Fourier transform of $f$ with respect to $t\in \R$, that is,
\textcolor{rot}{\ben
\textcolor{rot}{\hat{f}(\omega)=\mathcal{F}_{t\rightarrow \omega}[f]:=\int_\R f(t)\exp(i\omega t)\,dt},\quad \omega\in \R.
\enn}
Denote by $\hat{G}=\hat{G}(x,\omega)$ the Fourier transform of $G(x,t)$ with respect to $t$,
and define the compressional and shear waves numbers $k_p$ and $k_s$ in the Fourier domain as
\ben
k_p:=\omega\sqrt{\rho/(\lambda+2\mu)},\quad
k_s:=\omega\sqrt{\rho/\mu}.
\enn
Then we find that
\ben
\mu \Delta \hat{G}(\cdot,\omega)e_j+(\lambda+\mu)\nabla (\nabla\cdot \hat{G}(\cdot,\omega)e_j)+\omega^2\rho \hat{G}(\cdot,\omega)e_j =-\delta(\cdot)e_j,\quad j=1,2,3
\enn
and
\be\label{Pi}
\hat{G}(x-y,\omega)=\frac{1}{\mu}\Phi_{k_s}(x,y) \textbf{I}+\frac{1}{\rho\omega^2}\, \grad_x\,\grad_x^\top\;\left[\Phi_{k_s}(x,y)-\Phi_{k_p}(x,y) \right],\quad x\neq y.
\en
Here $\Phi_k(x,y)=e^{ik|x-y|}/(4\pi|x-y|)$ ($k=k_p, k_s$) is the fundamental solution to the Helmholtz equation $(\Delta+k^2) u=0$ in $\R^3$.
By Lemma \ref{lem:1}, we may take the Fourier transform of $U(x,t)$ with respect to $t$. Consequently, it holds in the frequency domain that
\be\label{U-f}
\mu \Delta \hat{U}(x,\omega)+(\lambda+\mu)\nabla (\nabla\cdot \hat{U}(x,\omega))+\omega^2\rho \hat{U}(x,\omega)=-f(x)\hat{g}(\omega),\quad \omega\in \R.
\en
Corresponding to the representation of $U(x,t)$ in the time domain, we have in the Fourier domain that
\be\label{U-hat}
\hat{U}(x,\omega)=\int_{\R^3} \mathcal{F}[G(x-y,\cdot)\ast g(\cdot)]\;f(y)dy
=\hat{g}(\omega)\;\int_{\R^3} \hat{G}(x-y,\omega)f(y)\,dy,\quad x\in \R^3,\quad\omega\in \R^+.
\en
\textcolor{rot}{ Here $*$ denotes the convolution product with respect to the time variable.}
Note that $\hat{U}(x,-\omega)=\overline{\hat{U}(x,\omega)}$, since $U(x,t)$ is real valued.

\section{Unique determination of spatial functions}\label{sec:spatial}
In this section we are interested in the inverse source problem of recovering $f$ from the radiated wave field
$\{U(x,t): |x|=R, t>T\}$ for some $R>R_0$ and $T>T_0$ under the a prior assumption
that $g$ is given. We suppose that
\textcolor{rot}{$f\in (L^2(\mathbb R^3))^3$, supp$(f)\subset B_{R_0}$, $g\in \mathcal C_0([0,T_0])$.}
Since $f$ and $g$ have compact support,
the initial boundary value problem (\ref{lame}), (\ref{eq:12}) and (\ref{eq:13})
admits a unique solution
\textcolor{rot}{ $U\in \mathcal{C} (\mathbb R, H^1(B_R))^3\cap \mathcal C^1(\mathbb R, L^2(B_R))^3$
for any $R>0$.}
 Let $f_\alpha$ and $U_\alpha$ ($\alpha=p,s$) be specified as in (\ref{f}) and (\ref{Ud}),
 respectively. Our uniqueness results are stated as following.
\begin{theorem}\label{Theorem-1} (i) The data set
$\{U(x, t): |x|=R, t\in(0, T_s)\}$ uniquely determines the spatial function $f$.
(ii) The data set of pure P- and S-waves, $\{U_\alpha(x, t): |x|=R, t\in(0, T_\alpha)\}$, uniquely determines $f_\alpha$ ($\alpha=p,s$).
\end{theorem}
\textcolor{rot2}{ We  remark that, since the measurement surface is spherical, the compressional and shear components $U_\alpha(x,t)$ ($\alpha=p,s$) can be decoupled
from the whole wave fields $U(x,t)$ on $|x|=R$. In fact, in the Fourier domain, $\hat{U}_\alpha(x,\omega)$ can be decoupled from $\hat{U}(x,\omega)$ on $|x|=R$ for every
fixed $\omega\in \R^+$; see e.g., \cite{BHSY} or Section \ref{test:spatial} in the 2D case. Hence the decoupling in the time domain can be achieved via Fourier transform.}
Below we present a frequency-domain approach and a time-domain approach to the proof of Theorem \ref{Theorem-1}.

\subsection{ Frequency-domain approach}\label{sec:Fouier}
{\bf Proof of Theorem \ref{Theorem-1}.}
(i) Assuming that $U(x, t)\textcolor{rot}{=} 0$ for all $|x|=R$ and $t\in(0, T_s)$, we need to prove that $f\equiv 0$ in $B_{R_0}$.
Recalling Lemma \ref{lem:1}, we have $U_\alpha(x, t)\textcolor{rot}{=} 0$ for all $|x|=R$, $t\in\R^+$.
\textcolor{rot}{ Combining this with the fact that $U(x,t)=0$, $(x,t)\in\mathbb R^3\times(-\infty,0]$,
we deduce that $U_\alpha(x, t)= 0$ for all $|x|=R$, $t\in\R$.
Then, applying the Fourier transform in time to $U_\alpha(x,\cdot)$ gives
$$\textcolor{rot1}{\hat{U}_\alpha(x, \omega)=\int_\R U_\alpha(x,t)e^{i\omega t}dt}=0,\quad\mbox{for all}\quad |x|=R,\ \omega\in\R^+.$$} Introduce the functions
\ben
v_p(x,\omega):=d\, e^{-i k_p d\cdot x},\quad v_s(x,\omega):=d^\perp\, e^{-i k_s d\cdot x},\quad d\in \mathbb{S}^2:=\{x\in \R^3: |x|=1\},
\enn
where $k_\alpha=k_\alpha(\omega)$ ($\alpha=p,s$) are the compressional and shear wave numbers, respectively, and
$d^\perp\in \mathbb{S}^2$ stands for a unit vector that orthogonal to $d$. Physically, $v_p$ and $v_s$
denote the compressional and shear plane waves propagating along the direction $d$, respectively.
They fulfill the time-harmonic Navier equation as follows
\ben
\mu \Delta v_\alpha+(\lambda+\mu)\nabla (\nabla\cdot v_\alpha)+\omega^2\rho v_\alpha=0,\quad \alpha=p,s.
\enn

Multiplying $v_\alpha$ to (\ref{U-f}) and applying Betti's formula to $\hat{U}$ and $v_\alpha$ in $B_R$, we obtain
\ben
-\hat{g}(\omega)\int_{B_R} f(x)\cdot v_\alpha (x,\omega)\,dx=\int_{|x|=R} \left[ T_{\nu} \hat{U}(x,\omega)\cdot v_\alpha(x,\omega)
-T_{\nu}v_\alpha(x,\omega)\cdot \hat{U}(x,\omega)\right]\,ds,
\enn
where $\nu=(\nu_1,\nu_2,\nu_3)^\top\in \mathbb{S}^2$ is the normal direction on $|x|=R$ pointing into $|x|>R$ and
$T_\nu=T_\nu^{(\lambda,\mu)}$ is the traction operator defined by
\be\label{stress-3D}
T_{\nu}\hat{U}:=2 \mu \, \partial_{\nu} \hat{U} + \lambda \,
\nu \, \divv \hat{U}+\mu \nu\times \curl \hat{U}.
\en
It follows from (\ref{U-hat}) that $\hat{U}(x,\omega)$ satisfies the Kupradze radiation when $|x|\rightarrow\infty$.
By well-posedness of the Dirichlet boundary value problem for the time-harmonic Navier system in $|x|>R$,
we obtain $T_{\nu}\hat{U}(x,\omega)\equiv0$ for all $|x|=R$, $\omega\in\R^+$. This also follows from the well-defined
Dirichlet-to-Neumann operator applied to $\hat{U}|_{|x|=R}$ for fixed $\omega\in \R^+$; see e.g., \cite{BHSY}.
Hence,
\be\label{eq:1}
\hat{g}(\omega)\int_{B_R} f(x)\cdot v_\alpha (x,\omega)\,dx=0,\quad \mbox{for all}\quad \omega\in \R^+.
\en
 \textcolor{rot2}{On the other hand,  applying integration by parts we get
\ben
\int_{\R^3} \nabla  \times  f_s(x)\cdot de^{-ik_p x\cdot d} dx=
-\int_{\R^3}   f_s(x)\cdot \nabla  \times (de^{-ik_p x\cdot d}) dx=0,
\enn
in which the boundary integral over $\partial B_R$ vanish due to the compact support of $f$ in $B_{R_0}\subset B_R$.}
It then follows
\ben
&\quad&\int_{B_R} f(x)\cdot v_p (x,\omega)\,dx\\
&=&\int_{\R^3} \nabla  f_p(x)\cdot d\,e^{-ik_s x\cdot d} \,dx+\int_{\R^3} \nabla\times  f_s(x)\cdot d\,e^{-ik_s x\cdot d} \,dx\\
&=&ik_s(2\pi)^{\frac{3}{2}}\;\hat{f_p}(k_pd).
\enn
\textcolor{rot2}{Note that here $\hat{f_p}$ refers to the Fourier transform of $f_p$ with respect to spatial variables, given by
$$ \hat{f_p}(\xi):=\int_{\R^3}f(x)e^{ix\cdot\xi}dx,\quad \xi\in\R^3.$$}
In the same way, we have
\ben\int_{\R^3} \nabla  f_p(x)\cdot d^\perp\,e^{-ik_s x\cdot d}\,dx
=-\int_{\R^3}   f_p(x)\;\nabla\cdot (d^\perp\,e^{-ik_s x\cdot d})\,dx
=ik_s(d\cdot d^\perp)\int_{\R^3}   f_p(x)\,e^{-ik_s x\cdot d}\,dx=0\enn
and we find
\ben\int_{B_R} f(x)\cdot v_s (x,\omega)\,dx=\int_{\R^3} \nabla\times f_s(x)\cdot d^\perp\,e^{-ik_s x\cdot d} \,dx
=ik_s(2\pi)^{\frac{3}{2}}\,\hat{f_s}(k_sd)\cdot (d\times d^\perp).
\enn
 Therefore, it follows from (\ref{eq:1}) that
\ben
\hat{f_p}(k_pd)=\hat{f_s}(k_sd)\cdot (d\times d^\perp)=0
\enn for all $d\in \mathbb{S}^2$ and for all $\omega\in\{\omega\in \R^+ : \hat{g}(\omega)\neq 0\}$. Since $g\neq 0$, one can always find an interval $(a,b)\subset\R^+$ such that $\hat{g}(\omega)\neq 0$ for $\omega\in(a,b)$. By the analyticity of $\hat{f}_\alpha$ ($\alpha=p,s$) and the arbitrariness of $d\in\mathbb{S}^2$, we finally obtain $\hat{f}_\alpha\equiv 0$. Applying inverse Fourier transform we get $f_\alpha\equiv 0$, implying that $f\equiv 0$.

(ii) By (\ref{eq:3}), the P- and S-waves fulfill the wave equations

\ben
&&\frac{1}{c_p^2}\,\partial_{tt}\,U_{p}(x,t)-\Delta U_p(x,t)=\frac{1}{\gamma_p}\,\nabla f_p(x)g(t),\\
&&\frac{1}{c_s^2}\,\partial_{tt}\,U_{s}(x,t)-\Delta U_s(x,t)=\frac{1}{\gamma_s}\,\nabla\times f_s(x)\,g(t),
\enn
in $\R^3\times(0,+\infty)$ together with the zero initial conditions at $t=0$, where $c_\alpha$ and $\gamma_\alpha$ ($\alpha=p,s$)
are given in (\ref{2.3}).
Applying  Duhalme's principle and Kirchhoff's formula for wave equations,
we can represent these P and S-waves as
\ben
U_p(x,t)=\frac{1}{4\pi\gamma_p}\int_{|y-x|\leq c_p(t+T_0)} \frac{\nabla f_p(y)g(t-|y-x|/c_p)}{|y-x|}dy,\\
U_s(x,t)=\frac{1}{4\pi\gamma_s}\int_{|y-x|\leq c_s(t+T_0)} \frac{\nabla\times f_s(y)g(t-|y-x|/c_s)}{|y-x|}dy,
\enn
for $x\in\R^3$, $t>0$. As done for the Navier equation in the proof of Lemma \ref{lem:1}, one can show that $U_\alpha(x,t)= 0$ for all $x\in B_R$, $t>T_\alpha$ ($\alpha=p,s$). Hence, the relation $U_\alpha(x,t)= 0$ for $x\in B_R$, $t\in(0,T_\alpha)$ would imply the vanishing of $U_\alpha(x,t)$ over $B_R$ for all $t\in \R^+$. Now, repeating the argument in the proof of the first assertion we deduce that
\ben
\nabla f_p=0,\quad \nabla\times f_s=0,\quad \divv f_s=0\quad\mbox{in}\quad B_R.
\enn
This implies that $f_p\equiv 0$, since $f_p\in H^1(B_R)$ and $f_p=0$ in $B_R\backslash B_{R_0}$.
To prove the vanishing of $f_s\in (L^2(B_R))^3$,
we apply the Helmholtz decomposition to $f_s$, i.e., $f_s=\nabla h_p+\curl\,h_s$, where $h_p\in H^1(B_R)$ and $h_s\in H_{\curl}(B_R)$
are compactly supported in $B_R$. Then it follows that $h_p=h_s\equiv 0$ in $B_R$ and thus $f_s\equiv 0$ in $B_R$;
see the proof of Lemma \ref{HD-uniqueness} in the Appendix.
\hfill $\Box$
\begin{remark}
The above proof of Theorem \ref{Theorem-1} by using Fourier transform is valid in odd dimensions only. The vanishing of the wavefields on $|x|=R$ for $t>T_s$ can be physically interpreted by Huygens' Principle, which however does not hold
 when the number of spatial dimensions is even.
The frequency-domain approach applies to two dimensions if we know the time-domain data for all $0<t<\infty$.
\end{remark}

As a bi-product of the frequency-domain approach to the proof of Theorem \ref{Theorem-1}, we show uniqueness
in recovering the source term of the time-dependent Schr\"odinger equation:
\be\label{W}
\left\{\begin{array}{lll}
i\hbar \partial_t W(x,t)=[-\frac{\hbar^2}{2\mu}\Delta+q(x)]W(x,t)+f_0(x)g_0(t)&&\mbox{in}\quad \R^3\times (0,+\infty)\\
W(x,0)=0&&\mbox{on}\quad \R^3,
\end{array}\right.
\en
where $\hbar$ is the reduced Planck constant, $\mu$ is the particle's reduced mass and $q$ is the particle's potential energy which is assumed to be
time-independent. Similar to the Lam\'e system, we shall assume that $f_0\in L^2(\R^3)$, $\mbox{supp}(f_0)\subset B_{R_0}$,
$g_0\in H^1_0(0,T_0)$.
The potential is supposed to be a real-valued \textcolor{rot}{nonnegative} function with compact support on $\overline{B}_R$
 for some $R>R_0$. The number
$\omega\in \C$ is called a Dirichlet eigenvalue of the operator $L_{\tilde{q}}:=\Delta-\tilde{q}$ with
$\tilde{q}(x):=2\mu/\hbar^2 q(x) $ if there exists a non-tirival function $V \in (H^1_0(B_R))^2$ such that
\ben
(L_{\tilde{q}}+\omega )V=0\quad\mbox{in}\quad B_R.
\enn
It can be easily  proved that the set of Dirichlet eigenvalues is discrete, which we donote by $\{\omega_n\}_{n=1}^\infty$, and that each eigenvalue is positive.
According to   \cite[Theorem 10.1, Chapter 3]{LM} and \cite[Remark 10.2, Chapter 3]{LM},
the initial problem \eqref{W} admits a unique solution
$W\in \mathcal C([0,+\infty);H^1(\R^3))\cap \mathcal C^1([0,+\infty);H^{-1}(\R^3))$.
Therefore, we can introduce the data $\{W(x,t): |x|=R, t\in \R^+\}$, for $W$
the unique solution of \eqref{W}.
\textcolor{rot2} { The following result extends the uniqueness proof of the inverse source problem for the Helmholtz equation \cite{EV} to the case of
time-dependent Lam\'e system with an inhomogeneous time-independent potential function. }
\begin{corollary}
Assume that $q\in\mathcal{C}_0(B_R)$ is known and that $\hat{g}_0(\omega_n')\neq 0$, $\omega_n'=\omega_n \hbar/(2\mu)$, for all $n=1,2,\cdots$. Then the data set
$\{W(x,t): |x|=R, t\in \R^+\}$ uniquely determines $f_0$.
\end{corollary}
\begin{proof} We assume that
\be\label{assu1} W(x,t)=0,\quad |x|=R,\ t\in[0,+\infty).\en
Since $g\in H^{1}_0(0,T)$, the extension of $W$ by $0$ on $\R^3\times(-\infty,0]$, is the unique solution of
$$\left\{\begin{array}{lll}
i\hbar \partial_t W(x,t)=[-\frac{\hbar^2}{2\mu}\Delta+q(x)]W(x,t)+f_0(x)g_0(t)&&\mbox{in}\quad \R^3\times \R \\
W(x,0)=0&&\mbox{on}\quad \R^3.
\end{array}\right.$$
Thus, without lost of generality we can assume that the solution of \eqref{W} is the solution of the problem on $\R^3\times \R$. Then, condition \eqref{assu1} implies that
\be\label{assu2} W(x,t)=0,\quad |x|=R,\ t\in\R.\en
\textcolor{rot2}{According to the estimate (10.14) in the proof of}
\cite[ Theorem 10.1, Chapter 3]{LM} we have
\ben\|W(\cdot,t)\|_{H^1(\R^3)}^2&\leq& C\int_0^{+\infty} (|g_0(s)|^2+|\mbox{d}g_0(s)/\mbox{d}s|^2)\,\|f_0\|_{L^2(\R^3)}^2\,ds\\
&\leq& C\|g_0\|_{H^1(0,T_0)}^2\|f_0\|_{L^2(\R^3)}^2,
\enn
for $t\in[0,+\infty)$,
where $C>0$ is a constant independent of $t$.
In particular, this estimate and the fact that $W(x,t)=0$ for $(x,t)\in\R^3\times(-\infty,0]$, proves that $W\in L^\infty(\R;H^1(\R^3))\subset \mathcal S'(\R;H^1(\R^3))$.
Therefore, we can  apply the Fourier transform $\mathcal{F}_{t\rightarrow \omega}$ to $W$  and deduce from (\ref{W}) that $\hat{W}=\mathcal{F}_{t\rightarrow \omega}W\in \mathcal S'(\R;H^1(\R^3))$ satisfies
 \be\label{L}
 L_{\tilde{q}} \hat{W}(x,\omega)+ \eta_1 \omega\, \hat{W}(x,\omega)=\eta_2\,f_0(x)\,\hat{g}_0(\omega),
 \quad x\in \R^3,\;\omega\in \R^+,
 \en
 with $\eta_1=2\mu/\hbar, \eta_2=2\mu/\hbar^2$. Note that the identity \eqref{L} is considered in the sense of distribution with respect to
 $(x,\omega)\in \R^3\times \R^+$. \textcolor{rot}{In view of \eqref{L}, we have $\Delta \hat{W}\in \mathcal S'(\R;L^2(\R^3))$ which implies that $\hat{W}\in \mathcal S'(\R;H^2(\R^3))$.}
 \textcolor{rot} {The equation (\ref{L}) can be rewritten as
 \ben
\Delta \hat{W}(x,\omega)+ k^2\, \hat{W}(x,\omega)=\eta_2\,f_0(x)\,\hat{g}_0(\omega)+
\tilde{q}(x)\, \hat{W}(x,\omega),\quad k:=\sqrt{\eta_1\omega}.
 \enn}
 \textcolor{rot}{Recalling Green's formula, for any $R_1>R$ we may represent $\hat{W}$ as the integral equation
  \ben
 \hat{W}(x,\omega)&=&
 \int_{\partial B_{R_1+1}}\left[ \partial_\nu \hat{W}(y,\omega) \Phi_k(x-y)-\partial_\nu \Phi_k(x-y)\,\hat{W}(y,\omega)\right]\,ds(y)\\
 &&-\int_{\R^3} \Phi_k(x-y)\,\tilde{q}(y)\,\hat{W}(y,\omega)\,dy
 -\eta_2\,\hat{g}_0(\omega)\int_{\R^3}\Phi_k(x-y)f_0(y)dy
  \enn
 for $x\in B_{R_1}$, where  $\Phi_k$ is the fundamental solution to the Helmholtz equation $(\Delta+k^2) u=0$.
On the other hand, we have
$$\begin{array}{l}  \left|\int_{\partial B_{R_1+1}}\left[ \partial_\nu \hat{W}(y,\omega) \Phi_k(x-y)-\partial_\nu \Phi_k(x-y)\,\hat{W}(y,\omega)\right]\,ds\right|\\
\leq C(\|\partial_\nu \hat{W}(y,\omega)\|_{L^2(\partial B_{R_1+1})}+\|\hat{W}(y,\omega)\|_{L^2(\partial B_{R_1+1})})\left(\int_{\partial B_{R_1+1}}|x-y|^{-2}ds(y)\right)^{\frac{1}{2}}\\
\leq C(\|\partial_\nu \hat{W}(y,\omega)\|_{L^2(\partial B_{R_1+1})}+\|\hat{W}(y,\omega)\|_{L^2(\partial B_{R_1+1})})\end{array}$$
and, since $\hat{W}(\cdot, \omega)\in H^2(\R^3)$, by density we deduce that
$$\lim_{R_1\to+\infty}(\|\partial_\nu \hat{W}(y,\omega)\|_{L^2(\partial B_{R_1+1})}+\|\hat{W}(y,\omega)\|_{L^2(\partial B_{R_1+1})})=0.$$
Therefore, sending $R_1\to+\infty$, we get
  \ben
 \hat{W}(x,\omega)=-\int_{\R^3} \Phi_k(x-y)\,\tilde{q}(y)\,\hat{W}(y,\omega)\,dy
 -\eta_2\,\hat{g}_0(\omega)\int_{\R^3}\Phi_k(x-y)f_0(y)dy,\quad x\in \R^3.
  \enn }
 This implies that
 $\hat{W}(\cdot,\omega)$  is the unique solution of (\ref{L})
satisfying  the Sommerfeld radiation condition when $|x|\rightarrow \infty$.
 Let $V_n\in (H_0^1(B_R))^2$ be an eigenfunction that corresponds to the Dirichlet eigenvalue $\omega_n$. Using the fact that $\hat{W}(\cdot,\omega)\in \{S\in H^1(B_R):\ \Delta S\in L^2(B_R)\}$ and
multiplying
 $V_n$ to
 both sides of (\ref{L}) with $\omega=\omega_n'$ and applying integral by parts, we obtain
 \ben
 \eta_1\,\hat{g}_0(w_n')\int_{B_R} f_0(x) V_n(x) dx&=&\left\langle \partial_{\nu} \hat{W}_n(\cdot ;\omega_n'),V_n\right\rangle_{H^{-\frac{1}{2}}(\partial B_R),H^{\frac{1}{2}}(\partial B_R)}-
 \int_{\partial B_R}\partial_\nu V_n(x)\, \hat{W}_n(x;\omega_n')ds(x) \\
 &=&\left\langle T_n[\hat{W}(\cdot ;\omega_n')],V_n\right\rangle_{H^{-\frac{1}{2}}(\partial B_R),H^{\frac{1}{2}}(\partial B_R)}-
 \int_{\partial B_R}\partial_\nu V_n(x)\, \hat{W}_n(x;\omega_n')ds(x),
 \enn
 where $T_n: H^{1/2}(\partial B_R)\rightarrow H^{-1/2}(\partial B_R)$ is the Dirichlet-to-Neumann map for radiating solutions to the Helmholtz equation
 $(\Delta+(\omega_n')^2)u=0$
 which satisfies
 the Sommerfeld radiation condition at infinity. In view of \eqref{assu2}, the fact that
 $\hat{g}_0(w_n')\neq 0$ and the fact that $T_n$
 is a linear bounded map, for every $n=1,2,\cdots$, we deduce, from the previous identity  that
 \ben
 \int_{B_R} f_0(x) V_n(x) dx=0,\quad n=1,2,\cdots.
 \enn
 Since the set of the Dirichlet eigenfunctions is complete over $(L^2(B_R))^2$, we conclude that
 the temporal function $f_0$ can be uniquely determined by the data. This finishes the uniqueness proof.
\end{proof}

\subsection{Time-domain approach}
In this subsection we present a time-domain \textcolor{rot}{proof of Theorem  \ref{Theorem-1}}.
\textcolor{rot}{Note that this demonstration can be extended to dimension two provided that we replace the data
$\{U(x, t):\ |x|=R,\  t\in(0, T_s)\}$ by $\{U(x, t):\ |x|=R,\  t\in(0, +\infty)\}$,
since Lemma \ref{lem:1} does not hold in dimension two}.

{\bf Proof of Theorem \ref{Theorem-1}}.
(i) Again we assume that $U(x, t)= 0$ for all $|x|=R$, $t\in(0, T_s)$ and\textcolor{rot}{, in view of Lemma \ref{lem:1}, this implies that $U(x, t)= 0$ for all $|x|=R$, $t\in(0, +\infty)$}. Our aim is to deduce that $f\equiv 0$.
Since the temporal function $g$ is known,
we apply Duhalme's principle to $U$ by setting
\be\label{Duhalme}
U(x,t)=\int_0^t V(t-s,x)g(s)\,ds,\quad x\in \R^3, t>0.
\en
The function $V$ then fulfills the homogeneous Lam\'e equation with non-zero initial conditions
\ben
&&\partial_{tt}V(x,t)=-c_p^2\;\nabla\times \nabla\times V(x,t)+c_s^2\;\nabla (\nabla\cdot V(x,t)),\\
&& V(x,0)=0,\quad \partial_t V(x,0)=f(x).
\enn
Further, we can continue $V$ onto $\R^3\times (-\infty, 0)$ preserving the Lam\'e equation and the initial conditions.
Since \textcolor{rot}{$g(t)=0$} for $t<0$, the function $t\rightarrow U(t,x)$, given by (\ref{Duhalme}), can be regarded as the convolution of $V(x,\cdot)\chi(\cdot)$ and $g(\cdot)\chi(\cdot)$, i.e.,
\be\label{eq:14}
U(x,t)=[V(x,t)\chi(t)]\ast [g(t)\chi(t)],
\en
 where $\chi$ is the characteristic function of $(0,\infty)$.
 By Lemma \ref{lem:1}, $U(x,t)=0$ for $|x|=R$ and $t\in \R$. Taking the Fourier transform to (\ref{eq:14}), we see
 \ben
 0=\mathcal{F}_{t\rightarrow \omega}[V(x,t)\chi(t)]\mathcal{F}_{t\rightarrow \omega}[g(t)\chi(t)]
 =\mathcal{F}_{t\rightarrow \omega}[V(x,t)\chi(t)]\,\hat{g}(\omega),\quad |x|=R.
 \enn
  Making use of the analyticity of $\mathcal{F}_{t\rightarrow \omega}[V(x,t)\chi(t)]$ with respect to $\omega$ and that  the fact that $g$ does not vanish identically, we deduce that
  $V(x,t)=0$ for $|x|=R$ and $t\in \R$.

We decouple $V$ into the sum of the compressional part $V_p$ and shear part $V_s$:
\ben
V=V_p+V_s,\quad V_p=\nabla v_p+\nabla\times v_s,\quad \nabla\cdot v_s=0\quad \mbox{in}\quad \R^3,
\enn
where $V_\alpha$ ($\alpha=p,s$) fulfills the homogeneous wave equation
\ben
\partial_{tt}V_\alpha(x,t)=c_\alpha^2\,\Delta V_\alpha\qquad\mbox{in}\quad B_R
\enn
and the initial conditions
\ben
V_\alpha(x,0)=0,\quad \partial_t V_p (x,0)=\nabla f_p(x),\quad
\partial_t V_s (x,0)=\nabla\times f_s(x).
\enn
Since $\mbox{Supp}(f)\subset B_{R_0}\subset B_R$, $V(x,t)$ has zero initial conditions in the unbounded domain $|x|>R$.
Consequently, we get $V\equiv 0$ for all $|x|>R$ and $t\in \R$, due to
the unique solvability of the hyperbolic system in $|x|>R$ with the Dirichlet boundary condition at $|x|=R$ for all $t>0$.
By uniqueness of the Helmholtz decomposition, it follows that  $V_\alpha=0$ in $|x|>R$ for all $t\in \R$.
In view of the unique continuation for the homogeneous wave equation (see e.g., \cite{ACY,RZ,Tataru}),
it can be deduced that $V_\alpha(x,t)=0$ in $B_R\times\R$, implying that $V=0$ for $x\in B_R$ and $t\in \R$.
In particular, $\partial_tV(x,0)=f(x)=0$ for $x\in B_R$.

(ii) If $U_{\alpha}(x,t)=0$ for $|x|=R$, $t\in(0,T_\alpha)$, then we have $V_\alpha =0$ on $\{|x|=R\}\times \R$.
Repeating the arguments above, it follows that $V_\alpha(x,0)=0$ in $|x|>R$ for $t\in \R$.
As a consequence of the unique continuation we get $V_\alpha(x,0)=0$ in $B_R\times \R$. Setting $t=0$ we obtain  $f_\alpha=0$ for $\alpha=p,s$.
\hfill $\Box$

\begin{remark} We think that the frequency-domain and time-domain approaches presented above could also yield stability estimate of the spatial function in terms of the time-domain data $\{U(x,t): |x|=R, 0<t<T_s\}$.
 The terminal time $T_\alpha$ ($\alpha=p,s$) in Theorem \ref{Theorem-1} are optimal. Non-uniqueness examples can be readily reconstructed if the terminal time is less than $T_\alpha$.
\end{remark}

\section{Unique determination of temporal functions}\label{sec:temporal}

Given some $T>0$, we suppose that $g\in (L^2(0,T))^3$ is an unknown vector-valued temporal function and
that the spatial function $f$ is known to be compactly supported in $B_{R_0}$ for some $R_0>0$.
We consider the inverse problem of determining $g$ from observations of the solution of
\begin{equation}\label{eq}\left\{\begin{array}{ll}\rho\partial_{tt}U(x,t)=\nabla\cdot \sigma(x,t)+f(x)g(t),\quad & (x,t)\in\R^3\times(0,T),\\  U(x,0)=\partial_tU(x,0)=0,\quad &x\in\R^3,\end{array}\right.\end{equation}
at one fixed point $x_0\in\mbox{supp}(f)$ (i.e., interior observations) or at the subbounary $\Gamma\subset \partial B_R$
(i.e., partial boundary observations). In order to state rigorously our problem, we start by considering the regularity of \textcolor{rot}{this initial  value
problem} (\ref{eq}).
\begin{lemma}\label{ll1} Let $g\in (L^2(0,T))^3$ and let $f\in H^p(\R^3)$, with $p>5/2$ be supported on $B_R$ for some $R>R_0$. Then problem \eqref{eq} admits a unique solution $U\in \mathcal C([0,T];H^{p+1}(\R^3))^3\cap
 H^2((0,T);H^{p-1}(\R^3))^3$ satisfying
\begin{equation}\label{l1aa} \|U\|_{\mathcal C([0,T];H^{p+1}(\R^3))^3}+\|U\|_{H^2((0,T);H^{p-1}(\R^3))^3}\leq C\|g\|_{L^2(0,T)^3}\|f\|_{H^p(\R^3)},\end{equation}
with $C>0$  depending on $\rho$, $\lambda$, $\mu$, $R$.
\end{lemma}
\begin{proof}

Applying Fourier transform to $U(\cdot, t)$ with respect to spatial variables, denoted by $\hat{U}$, we find
\be\label{eq:4}\begin{aligned}
&\partial_{tt} \hat{U}(\xi,t)+A(\xi)\hat{U}(\xi,t)=\frac{g(t)\hat{f}(\xi)}{\rho}\quad\mbox{in}\quad \R^3\times(0,T),\\
&\hat{U}(\xi,0)=0,\quad \partial_t \hat{U}(\xi,0)=0,\quad \xi\in \R^3,
\end{aligned}
\en where the matrix $A(\xi)\in \R^{3\times 3}$ is defined by
\ben
A(\xi):=\frac{\mu}{\rho} |\xi|^2\,\textbf{I}+\frac{(\lambda+\mu)}{\rho}\xi\otimes \xi, \quad \xi=(\xi_1,\xi_2,\xi_3)\in \R^3.
\enn Evidently, $A(\xi)$ is a real-valued symmetric matrix, with the eigenvalues given by  $\frac{(\lambda+2\mu)|\xi|^2}{\rho}$, $\frac{\mu|\xi|^2}{\rho}$, $\frac{\mu|\xi|^2}{\rho}$; see Lemma \ref{Lem:eigenvalue} in the Appendix.
Denote by $A^{1/2}(\xi)$ the square roof of $A(\xi)$ and by  $A^{-1/2}(\xi)$ the inverse of $A^{1/2}(\xi)$. Then the unique solution to (\ref{eq:4}) takes the form
\begin{equation}\label{Fou}
\hat{U}^T(\xi,t)=\int_0^t g^T(s)\, A^{-1/2}(\xi)\,\sin \left( A^{1/2}(\xi) (t-s) \right)\,\frac{\hat{f}(\xi)}{\rho}\,ds.
\end{equation}
On the other hand,   for all $t\in[0,T]$ and $s\in[0,t]$, fixing
$$H(t-s,\cdot):=\xi\mapsto A^{-1/2}(\xi)\,\sin \left( A^{1/2}(\xi) (t-s) \right)\,\frac{\hat{f}(\xi)}{\rho},$$
we have
\begin{equation}\label{l1a}\begin{aligned}\|H(t-s,\cdot)\|_{L^2(\R^3)^{ 3\times3}}^2&\leq \mu^{-1}\rho^{-1}\textcolor{rot}{\|\hat{f}\|_{L^\infty(\R^3)}^2}\int_{B_1}|\xi|^{-2}d\xi+4\mu^{-1}\rho^{-1}\int_{\R^3\setminus B_1}(1+|\xi|^2)^{-1}|\hat{f}(\xi)|^2d\xi\\
\ &\leq C\mu^{-1}\rho^{-1}\textcolor{rot}{|B_R|\|\hat{f}\|_{L^2(\R^3)}^2}+4\mu^{-1}\rho^{-1}\|f\|_{L^2(\R^3)}^2,\end{aligned}\end{equation}
with $C$ a constant. Note that here we use the fact that $\xi\mapsto |\xi|^{-2}\in L^1(B_1)$, since $2<3$,  and the fact that supp$(f)\subset B_R$. Moreover, we apply the fact that $\lambda+\mu>0$ to deduce that $|A^{-1/2}(\xi)|\leq \rho^{\frac{1}{2}}\mu^{-1/2}|\xi|^{-1}$. In the same way, we have
\begin{equation}\label{l1b}\||\xi|^{p+1}H(t-s,\cdot)\|_{L^2(\R^3)^{ 3\times3}}^2\leq \mu^{-1}\rho^{-1}\int_{\R^3}|\xi|^{2p}|\hat{f}(\xi)|^2d\xi\leq \mu^{-1}\rho^{-1}\textcolor{rot}{\|f\|_{H^p(\R^3)}^2}.\end{equation}
Combining estimates \eqref{l1a}-\eqref{l1b}, one can easily deduce that $U\in \mathcal C([0,T]; H^{p+1}(\R^3))^3$. In the same way, we have
\begin{equation}\label{l1c}\|(1+|\xi|^2)^{\frac{p-1}{2}}\partial_tH(t-s,\cdot)\|_{L^2(\R^3)^{3\times3}}
+\|(1+|\xi|^2)^{\frac{p-1}{2}}\partial_t^2H(t-s,\cdot)\|_{L^2(\R^3)^{ 3\times3}}\leq C\|f\|_{H^p(\R^3)},\end{equation}
where $C$ depends on $\rho$, $\lambda$, $\mu$, $R$.
Moreover, for almost every $\xi\in\R^3$, we have
$$\hat{U}^T(\xi,\cdot):t\mapsto\hat{U}^T(\xi,t)\in H^2(0,T),$$
with
$$\partial_t\hat{U}^T(\xi,t)=\int_0^tg(s)^T\partial_tH(t-s,\xi)ds,\quad
\partial_{tt}\hat{U}^T(\xi,t)=\frac{g(t)^T\hat{f}(\xi)}{\rho}+\int_0^tg(s)^T\partial_{tt}H(t-s,\xi)ds.$$
Combining this with \eqref{l1c},  we deduce that
$U\in H^2((0,T); H^{p-1}(\R^3))^3$ and we deduce \eqref{l1aa} from the previous estimates.

\end{proof}

 According to Lemma \ref{ll1} and the Sobolev embedding theorem we have
 $U\in \mathcal C([0,T];\mathcal C^2(\R^3))^3\cap H^2((0,T); \mathcal C(\R^3))^3$ and the trace $t\mapsto U(x_0,t)$, for some point $x_0\in \R^3$, is well defined as an element of
 $H^2((0,T))^3$.
Below we consider the inverse problem of determining the evolution function $g(t)$ from the interior observation of the wave fields
$U(x_0, t)$ for $t\in(0,T)$ and some $x_0\in\mbox{supp}(f)$.

\begin{theorem}[Uniqueness and stability with interior data]\label{Th:in}
Let $x_0\in B_R$, $p>5/2$ and consider $M,\delta>0$ such that
$$ \mathcal A_{x_0,p,\delta,M}:=\{h\in H^p(\R^3):\ \|h\|_{H^p(\R^3)}\leq M,\ |h(x_0)|\geq\delta,\ \textrm{supp}(h)\subset B_R\}\neq\emptyset.$$
Then,
for $f\in \mathcal A_{x_0,p,\delta,M}$, it holds that
$$\|g\|_{L^2(0,T)^3}\leq C\, \|\partial_{tt}U(x_0,\cdot)\|_{L^2(0,T)^3}$$
where $C$ depends on $\lambda$, $\mu$, $\rho$, $p$, $x_0$, $M$, $R$, $\delta$ and $T$.
In particular, this estimate implies that the data $\{U(x_0,t): t\in(0,T)\}$  determines uniquely the temporal function $g$.
\end{theorem}
\begin{proof}
According to \eqref{Fou}, the solution $U$ of \eqref{eq} is given by
\ben U(x,t)^T=(2\pi)^{-3}\int_{\R^3}\left(\int_0^t g(s)^T\, A^{-1/2}(\xi)\,\sin \left( A^{1/2}(\xi) (t-s) \right)\,\frac{\hat{f}(\xi)}{\rho}\right)\,e^{i\xi\cdot x}\,d\xi,\quad (x,t)\in\R^3\times[0,T] \enn
and applying Fubini's theorem we find
\ben
U(x,t)^T=(2\pi)^{-3}\int_0^t g(s)^T\,\left(\int_{\R^3} A^{-1/2}(\xi)\,\sin \left( A^{1/2}(\xi) (t-s) \right)\,\frac{\hat{f}(\xi)}{\rho}\,e^{i\xi\cdot x}\,d\xi\right)\,ds,\quad (x,t)\in\R^3\times[0,T].
\enn
In particular, in view of Lemma \ref{ll1}, $U\in \mathcal C([0,T];H^{p+1}(\R^3))^3\cap H^2((0,T); H^{p-1}(\R^3))^3$ satisfies \eqref{l1aa}.
Further, direct calculations show that
\ben
 \mathcal{L}_{\lambda,\mu}U(x,t)^T=-(2\pi)^{-3}\int_0^t g(s)^T\,\left(\int_{\R^3} A^{1/2}(\xi)\,\sin \left( A^{1/2}(\xi) (t-s) \right)\,\frac{\hat{f}(\xi)}{\rho}\,e^{i\xi\cdot x}\,d\xi\right)\,ds.
\enn
Since $|A^{1/2}(\xi)\textbf{a}|\leq \frac{\sqrt{\lambda+2\mu}}{\sqrt{\rho}}|\xi| |\textbf{a}|$ for all $\textbf{a}\in \C^3$,
the previous identity can be estimated by
\be\nonumber
|\mathcal{L}_{\lambda,\mu}U(x,t)|&\leq& \frac{\sqrt{\lambda+2\mu}}{\sqrt{\rho}} \int_0^t |g(s)|\,ds\int_{\R^3} |\hat{f}(\xi)|\,|\xi|\,d\xi \\ \nonumber
&\leq& \frac{\sqrt{\lambda+2\mu}}{\sqrt{\rho}} \int_0^t |g(s)|\,ds \,
||\hat{f}(\xi)(1+|\xi|^2)^{p/2}||_{L^2(\R^3)}\, ||(1+|\xi|^2)^{(1-p)/2}||_{L^2(\R^3)}\\
\label{eq:5}
&\leq& M_0\,\frac{\sqrt{\lambda+2\mu}}{\sqrt{\rho}}\; ||f||_{H^p(\R^3)} \,\int_0^t |g(s)|\,ds
\en
where $M_0=||(1+|\xi|^2)^{(1-p)/2}||_{L^2(\R^3)}<\infty$.
Since $|f(x_0)|\geq \delta$, we derive from the governing equation of $U$ and (\ref{eq:5})  that
\ben
|g(t)|&=&\frac{1}{|f(x_0)|}\left| \rho \partial_{tt} U(x_0,t) -\mathcal{L}_{\lambda,\mu}U(x_0,t) \right|\\
&\leq& M_1\,|\partial_{tt} U(x_0,t)|+M_2\,\int_0^t |g(s)|\,ds
\enn for all $t\in(0,T)$, where $M_1=\rho/\delta$,
 $M_2=M_0\frac{\sqrt{\lambda+2\mu}}{\sqrt{\rho}}\,M /\delta$.
 Applying the Grownwall inequality stated in Lemma \ref{Gronwall}, for almost every $t\in(0,T)$, we find
$$
\begin{aligned}|g(t)| &\leq  M_1|\partial_{tt} U(x_0,t)|+\,M_1\,M_2\,\int_0^t|\partial_{tt} U(x_0,s)|\,e^{M_2(t-s)}\,ds\\
\ &\leq M_1\,|\partial_{tt} U(x_0,t)|\,+M_1\,M_2\,T^{\frac{1}{2}}\,e^{M_2T}\, \|\partial_{tt} U(x_0,\cdot)\|_{L^2(0,T)^3}.\end{aligned}$$
 Therefore, taking the norm $L^2(0,T)$ on both sides of the inequality, implies that
\ben
||g||_{L^2(0,T)^3}\leq (M_1+M_1\,M_2\,T\,e^{M_2T})\,\|\partial_{tt}U(x_0,\cdot)\|_{L^2(0,T)^3}.
\enn
This completes the proof.
\end{proof}

To state uniqueness with partial boundary measurement data, we need the concept of non-radiating source.
\begin{definition}\label{def}
 The compactly supported function $f$ is called a non-radiating source at the frequency $\omega\in \R^+$ to
 the Lam\'e system if the unique radiating solution to
 the inhomogeneous Lam\'e system
 \be\label{P}
 \mathcal{L}_{\lambda,\mu} u(x) +\omega^2\rho u(x)=f(x)\,P,\quad j=1,2,3,
 \en
 does not vanish identically in $\R^3\backslash \overline{\mbox{supp}(f)}$ for any $P\in \C^3$.
\end{definition}

\begin{theorem}[Uniqueness with partial boundary data]\label{TH:temporal} Suppose that $f\in L^2(B_R)$ is known to be a compacted supported function over
$B_{R_0}$ for some $R_0<R$  and that $f$ is not a non-radiating source for all $\omega\in \R^+$. Then the temporal function
 $g\in \mathcal{C}_0([0,T_0])^3$ can be uniquely determined by the partial boundary measurement data $\{U(x,t): x\in \Gamma, t\in(0,T_s)\}$
 where $\Gamma\subset \partial B_R$ is an arbitrary subboundary with positive Lebesgue measure and $T_s$ is defined in (\ref{Tps}).
\end{theorem}
\begin{proof}
 Let $w_j=w_j(x,\omega)$ ($j=1,2,3$) be the unique radiating solution to
 the inhomogeneous Lam\'e system
 \ben
\mathcal{L}_{\lambda,\mu} w_j(x) +\omega^2\rho w_j(x)=f(x)\,e_j,\quad j=1,2,3,
 \enn
 which does not vanish identically in $|x|\geq R$ by our assumption. Set the matrix $W:=(w_1,w_2,w_3)\in \C^{3\times 3}$.
 Then $W(\cdot, \omega)$ solves the matrix equation
 \ben
\mathcal{L}_{\lambda,\mu} W(x,\omega) +\omega^2\rho W(x,\omega)=f(x)\,\textbf{I}\quad\mbox{in}\quad \R^3\times (0,\infty).
 \enn
 Note that here the action of the differential operator is understood column-wisely, and $W$ can be represented as
 \ben
 W(x,\omega)=\int_{\R^3} \hat{G}(x-y, \omega) f(y)\,dy,\quad x\in \R^3,
 \enn where $\hat{G}$ is the Green's tensor to the time-harmonic Lam\'e system. In view of (\ref{U-hat}),
 the Fourier transform $\hat{U}(x,\omega)$ of $U(x,t)$ can be written as
 \be\label{UW}
 \hat{U}(x,\omega)=W(x,\omega)\,\hat{g}(\omega)\quad \mbox{for all}\quad \omega\in \R^+,\;|x|=R.
 \en
 We claim that for each $\omega_0\in \R^+$, there always exists $x_0\in \G\subset \partial B_R$ such that
 $\mbox{Det}(W(x_0,\omega_0))\neq 0$. Suppose on the contrary that $\mbox{Det}(W(x,\omega_0))= 0$ for all
 $x\in \G$.
 This implies that there exist $c_j\in \C^3$ such that
 \be\label{V}
V(x):= c_1w_1(x,\omega_0)+c_2w_2(x,\omega_0)+c_3w_3(x,\omega_0)=0\quad\mbox{on}\quad \Gamma.
 \en
 By the analyticity of $w_j$ in a neighborhood of $|x|=R$ and the analyticy of the surface $\Gamma\subset \partial B_R$,
 we conclude that (\ref{V}) holds  on $|x|=R$. By uniqueness of the exterior Dirichlet boundary value problem, we have $V(x)=0$
  in $|x|>R$, and by unique continuation it holds that $V(x)=0$ for all $x$ lying outside of the support of $f$.
 However, it is easy to observe that $V$ satisfies the inhomogeneous equation (\ref{P}) with $P=c_1e_1+c_2e_2+c_3e_3$,
which contradicts the fact that $f$ is not a non-radiating source. Therefore, by (\ref{UW}) we get
\ben
\hat{g}(\omega_0)=[W(x_0,\omega_0)]^{-1} \,\hat{U}(x_0,\omega_0)\in \C^{3\times 1}\quad\mbox{for some}\quad x_0\in \Gamma.
\enn
Note that $\omega_0$ is arbitrary and the point $x_0$ depends on $\omega_0$.
Hence, if $U(x,t)=0$ for all  $x\in \Gamma$ and $t\in(0,T_s)$, then $\hat{U}(x,\omega)=0$ for all $x\in \G$ and
$\omega\in \R^+$. This implies that $\hat{g}(\omega)=0$ for all $\omega\in \R$ and thus $g\equiv 0$.
\end{proof}

\section{Numerical experiments}\label{sec:numerics}
\textcolor{rot2} {In this section, we propose a Landweber iterative method for reconstructing the spatial function $f$ in 2D and
a non-iterative inversion scheme based on the proof of
Theorem \ref{TH:temporal} for recovering the temporal function $g$ in 3D.
Several numerical examples will be illustrated to examine the
effectiveness of the proposed methods.}

\subsection{Reconstruction of spatial functions}\label{test:spatial}

We consider the inverse source problem presented in Section \ref{sec:spatial}.
Our aim is to reconstruct the spacial function in two dimensions,
relying  on the Landweber iterative method for solving linear algebraic equations.
Assume that the time-dependent data $U(x,t), x\in\partial B_{R}$ ($R>R_0$) is measured over the time interval $[0,T]$
where $T>0$ is sufficiently large such that the integral
\ben
\textcolor{rot1}{\int_0^T\, U(x,t)\exp(i\omega t)\,dt}
\enn
can be used to approximate the Fourier transform $\hat{U}(x,\omega)$ for any $\omega\in\R^+$.
In the time-harmonic regime, it is supposed that the multi-frequency data
 $\hat{U}(x,\omega_k), x\in\partial B_{R}$ for $k=1,\cdots,K$ are available. Hence, the time-dependent inverse source problem can be transformed to a problem in the Fourier domain with near-field data of multi frequencies.
   In 2D, the Helmhotz decomposition of $\hat{U}$ takes the form
  $\hat{U}=\hat{U}_p+\hat{U}_s$, where the compressional part $\hat{U}_p$ and shear part $\hat{U}_s$ are given by
\be
\label{decomposition}
\hat{U}_p=-\frac{1}{k_p^2}\,\mbox{grad}\,\mbox{div}\;\hat{U},\quad \hat{U}_s=\frac{1}{k_s^2}\,\overrightarrow{\mbox{curl}}\,\mbox{curl}\;\hat{U}.
\en
Here the two-dimensional operators \mbox{curl} and $\overrightarrow{\mbox{curl}}$ are defined respectively by
\ben
\mbox{curl}\,v=\partial_1 v_2-\partial_2 v_1,\quad v=(v_1,v_2)^\top,\qquad \overrightarrow{\mbox{curl}}\; h:=(\partial_2h, -\partial_1h)^\top.
\enn
 Writing $\hat{u}_p:=-1/k_p^2\;\mbox{div}\,\hat{U}$ and $\hat{u}_s=1/k_s^2\;\mbox{curl}\,\hat{U}$, we have
 $\hat{U}=\mbox{grad}\,\hat{u}_p+\overrightarrow{\mbox{curl}}\,\hat{u}_s$ and the scalar functions $\hat{u}_\alpha$ ($\alpha=p,s$) satisfy
 the Sommerfeld radiation condition
\ben
\lim_{r \to \infty} \sqrt{r}\left(\frac{\partial \hat{u}_\alpha}{\partial r}-ik_\alpha \hat{u}_\alpha\right) = 0,\quad r=|x|,\quad \alpha=p,s
\enn
uniformly with respect to all $\hat{x}=x/|x|\in\mathbb{S}^{1}$.

For $|x|\ge R$, the radiation solutions $\hat{u}_\alpha$ can be expressed in terms of Hankel functions of the first kind,
\be\label{eq:8}
\hat{u}_\alpha(|x|,\theta)=\sum_{n\in\Z}\,\hat{u}_{\alpha,n}\,H_n^{(1)}(k_\alpha|x|)\exp(in\theta),\quad x=|x|(\cos\theta,\sin\theta), \;|x|\ge R.
\en
For every fixed $\omega\in \R^+$, the coefficients $\hat{u}_{\alpha,n}\in \C$ are uniquely determined by $\hat{U}(x,\omega)|_{|x|=R}$ as follows (see e.g., \cite{BHSY})
\be\label{eq:9}
\begin{pmatrix}
\hat{u}_{p,n} \\
\hat{u}_{s,n}
\end{pmatrix}=\frac{1}{2\pi R}[A_n(R)]^{-1}\int_0^{2\pi}\,\begin{pmatrix}
\cos\theta & \sin\theta \\
-\sin\theta & \cos\theta
\end{pmatrix}\,\hat{U}(R,\theta;\omega)d\theta,
\en
where
\be\label{eq:10}
A_n(R)=\begin{pmatrix}
t_p {H_n^{(1)}}'(t_p) & in H_n^{(1)}(t_s)\\
in H_n^{(1)}(t_p) & -t_s {H_n^{(1)}}'(t_s)
\end{pmatrix},\quad t_\alpha=k_\alpha R,\; \alpha=p,s.
\en
This means that, in the Fourier domain, the $P$ and $S$-waves can be decoupled from
the whole wave field $\hat{U}$ on $|x|=R$ for every fixed frequency $\omega$.

Below we shall consider the inverse problems of reconstructing $f_p$, $f_s$ and $f$ from the wave fields $\hat{u}_p(x,\omega)|_{\partial B_{R}}$, $\hat{u}_s(x,\omega)|_{\partial B_{R}}$ and $u(x,\omega)|_{\partial B_{R}}$ at a finite number of frequencies $\omega=\omega_k$, $k=1,\cdots,K$, respectively.
Recall from (\ref{U-hat}) that
\be
\label{eq:6}
\hat{U}(x,\omega)/\hat{g}(\omega)=\int_{B_R}\,\hat{G}(x-y)\,f(y)dy, \quad |x|=R,\quad \hat{g}(\omega)\neq 0,
\en
where $\hat{G}$ is the fundamental displacement tensor of the Navier equation of the form (\ref{Pi}) with the fundamental solution of the two-dimensional Helmholtz equation given by
\ben
\Phi_k(x,y)=\frac{i}{4}H_0^{(1)}(k|x-y|),\quad x\ne y,\quad x,y\in \R^2.
\enn
Analogously, the compressional and shear components of $\hat{U}$ can be represented by (cf. (\ref{eq:3}))
\be\label{eq:7}
\hat{u}_\alpha(x,\omega)/\hat{g}(\omega)=\frac{1}{\gamma_\alpha}\int_{B_R}\,\Phi_{k_\alpha}(x,y)\,f_\alpha(y)dy,\quad \alpha=p,s.
\en
Our numerical scheme relies on solvability of the ill-posed integral equations (\ref{eq:6}) and (\ref{eq:7}) for finding $f$ and $f_\alpha$.
 Since $f(x)$ is real-valued, it is more convenient to consider real-valued integral equations from numerical point of view. Taking the real and imaginary parts of (\ref{eq:6}) gives
\be
\label{Full-real}
\mbox{Re}\{\hat{U}(x,\omega)/\hat{g}(\omega)\}&=&\int_{B_R}\, \mbox{Re}\{\hat{G}(x-y, \omega)\}\,f(y)dy, \quad |x|=R,\\
\label{Full-imag}
\mbox{Im}\{\hat{U}(x,\omega)/\hat{g}(\omega)\}&=&\int_{B_R}\, \mbox{Im}\{\hat{G}(x-y,\omega)\}\,f(y)dy, \quad |x|=R.
\en
Furthermore, for the pressure part $\hat{u}_p$ and shear part $\hat{u}_s$, we have
\be
\label{ps-real}
\mbox{Re}\{\hat{u}_\alpha(x,\omega)/\hat{g}(\omega)\}&=&\frac{1}{\gamma_\alpha}\int_{B_R}\, \mbox{Re}\{\Phi_{k_\alpha}(x,y)\}f_\alpha(y)dy, \quad |x|=R,\\
\label{ps-imag}
\mbox{Im}\{u_\alpha(x,\omega)/\hat{g}(\omega)\}&=&\frac{1}{\gamma_\alpha}\int_{B_R}\, \mbox{Im}\{\Phi_{k_\alpha}(x,y)\}f_\alpha(y)dy, \quad |x|=R.
\en
The equations (\ref{Full-real})-(\ref{ps-imag}) are
 Fredholm integral equations of the first kind. These equations are ill-posed,  since the singular values of the matrix resulting from the discretized integral kernel are rapidly decaying.
Now, we describe a Landweber iterative method to solve the ill-posed integral equations (\ref{Full-real})-(\ref{ps-imag}). Consider the linear operator equations
\be
\label{OperatorEqn}
V_k\, (S)=v_k,\quad k=1,\cdots,K,\quad S=f, f_p, f_s,
\en
where $v_k=\hat{U}(x,\omega_k)$ or $v_k=\hat{u}_\alpha(x,\omega_k)$ denotes
the measurement data at the frequency $\omega_k$. We denote by
$S_{l,k}$ the inverse solution obtained at the $l$-th iteration step reconstructed from the data set at the frequency $\omega_k$.
\textcolor{rot1}{Due to the linearity of (\ref{OperatorEqn}), a straightforward Landweber iteration (see, e.g.,\cite{BLLT}) can be
applied as a regularization scheme for solving (\ref{OperatorEqn}).
For clarity We summarize the inversion process in Table \ref{table}}.

\begin{table}\caption{Landweber iterative method for reconstructing spatial functions.}\label{table}
\begin{tabularx}{\textwidth }{>{\bfseries}lX}
\toprule
Step 1 & Set an initial guess $S_{0,0}$  \\ \midrule
Step 2 & Update the source function $S$ by the iterative formula
    \ben
    S_{l,k}=S_{l-1,k}+\epsilon V_k^*\,(v_k-V_k\,S_{l-1,k}), \quad l=1,\cdots,L,
    \enn
    where $\epsilon$ and $L$ are the step length and total number of iterations, respectively. \\ \midrule
Step 3 & Set $S_{0,k+1}=S_{L,k}$ and repeat Step 2 until the highest frequency $\omega_K$ is reached.\\
\bottomrule
\end{tabularx}
\end{table}

\begin{figure}[htbp]
\centering
\begin{tabular}{cc}
\includegraphics[scale=0.3]{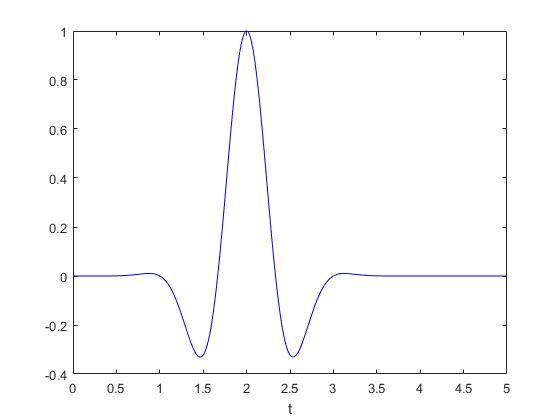} &
\includegraphics[scale=0.3]{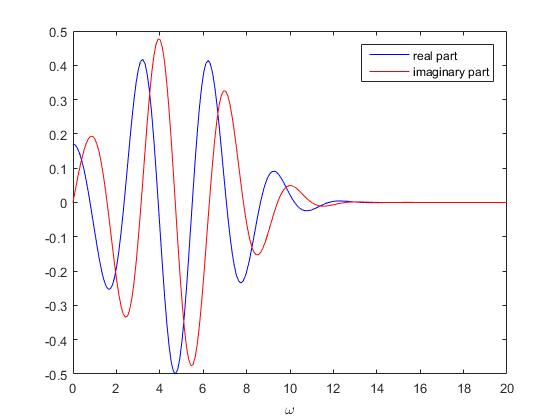} \\
(a) $g(t)$ & (b) $\hat{g}(\omega)$  \\
\end{tabular}
\caption{\textcolor{rot1}{The exact pulse function $g(t)$ and its Fourier transformation $\hat{g}(\omega)$.}}
\label{pulse}
\end{figure}

\begin{figure}[htbp]
\centering
\begin{tabular}{cc}
\includegraphics[scale=0.3]{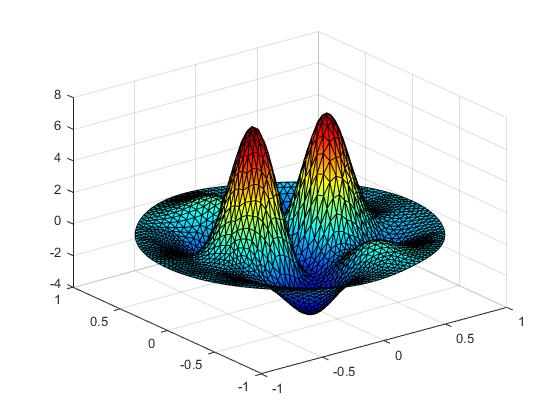} &
\includegraphics[scale=0.3]{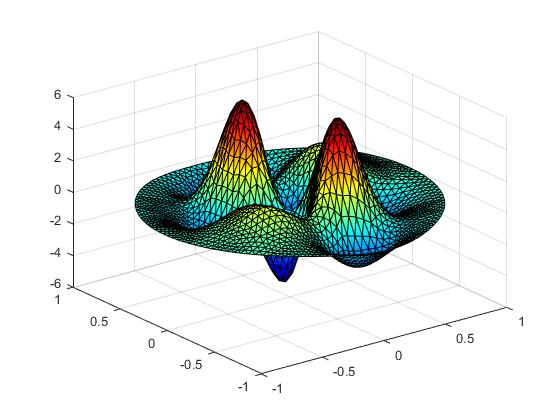} \\
(a) $f_1$ & (b) $f_2$ \\
\includegraphics[scale=0.3]{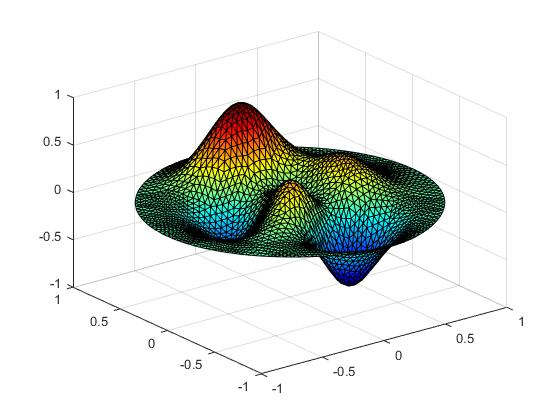} &
\includegraphics[scale=0.3]{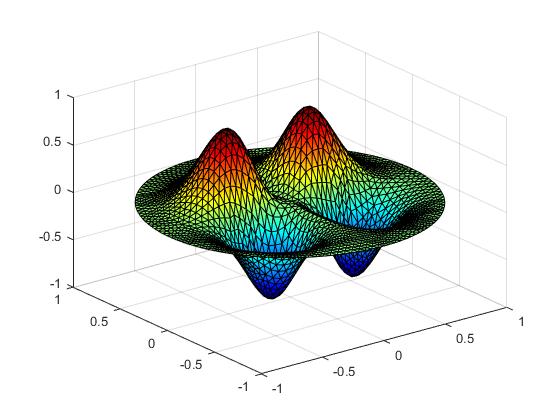} \\
(c) $f_p$ & (d) $f_s$ \\
\end{tabular}
\caption{The exact spatial source function $f=(f_1,f_2)$ and its compressional component $f_p$ and shear component $f_s$.}
\label{exactsource}
\end{figure}

Below we present several numerical examples to demonstrate the validity and effectiveness of the proposed method. In the following we always choose
\ben
\hat{g}(\omega)=\int_0^T\,g(t)\exp(i\omega t)dt,\quad g(t)=\begin{cases}
\cos(1.5\pi(t-t_0))\exp(-\pi(t-t_0)^2),& t\le T, \cr
0,& t>T,
\end{cases}
\enn
where $T=5$, $t_0=2$. The functions $\hat{g}$ and $g$ are plotted in Figure \ref{pulse}, which shows that $\hat{g}$ is nonzero in $(0,20)$. The source function $f$ in $B_R$ with $R=1$ is defined by
\ben
f=(f_1,f_2)^\top=\nabla\,f_p+\overrightarrow{\mbox{curl}}\,f_s,
\enn
where
\ben
f_p(x)&=&0.3(1-3x_1)^2\exp(-9x_1^2-(3x_2+1)^2) -(0.6x_1-27x_1^3-3^5x_2^5)\exp(-9x_1^2-9x_2^2)\\
&-&0.03\exp(-(3x_1+1)^2-9x_2^2),\\
f_s(x)&=&135x_1^2x_2\exp(-9x_1^2-9x_2^2);
\enn
see Figure \ref{exactsource}. We choose $\mu=1$, $\lambda=2$, $\rho=1$ and $R=2$. The scattering data is collected at 64 uniformly distributed points
on the circle $\partial B_{R}$. The total number of iterations is set to be $L=10$.

In the static case, we simulate the data $\hat{U}(x,\omega)$ by solving the inhomogeneous time-harmonic Navier equation using finite element method coupled with an exact transparent boundary condition. Then the compressional and shear parts, $\hat{u}_p$ and $\hat{u}_s$, are decoupled from $\hat{U}(x,\omega)$ via (\ref{eq:8})-(\ref{eq:10}). The near-field data of
twenty equally spaced frequencies from 1 to 20 are calculated. Figure \ref{TH-1} shows the reconstructed $S_1$ and $S_2$ from $\{\hat{U}(x,\omega_k): |x|=R, k=1,2,\cdots, 20\}$, while
 Figure \ref{TH-2} presents the reconstructed $f_p$ and $f_s$ from the counterpart of compressional and shear waves, respectively.

\begin{figure}[htbp]
\centering
\begin{tabular}{cc}
\includegraphics[scale=0.3]{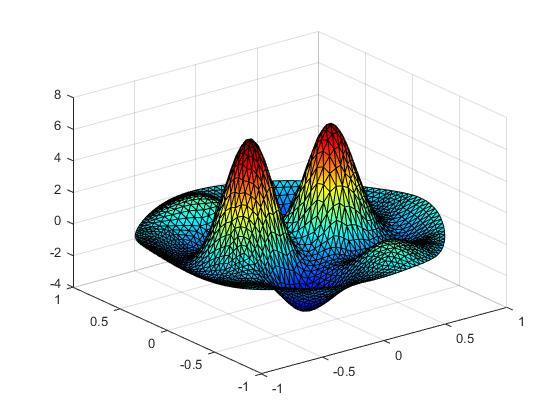} &
\includegraphics[scale=0.3]{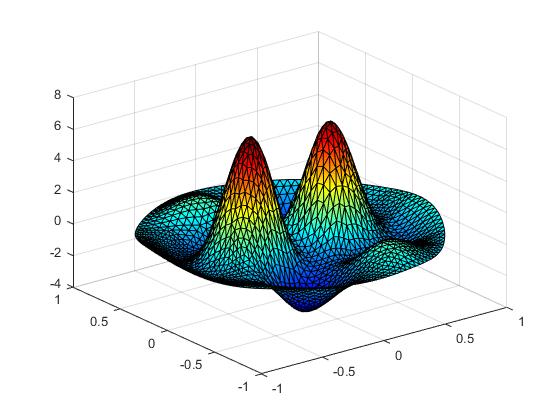} \\
(b) Reconstructed $f_1$ & (c) Reconstructed $f_1$ \\
\includegraphics[scale=0.3]{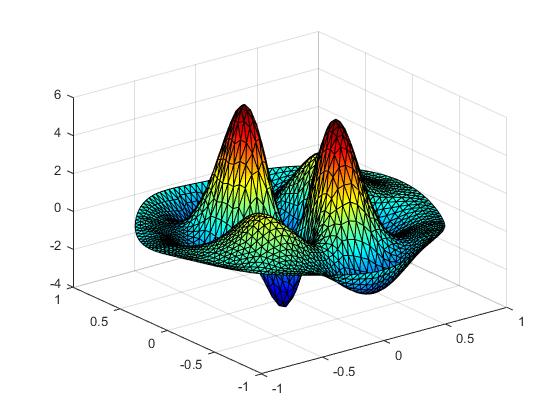} &
\includegraphics[scale=0.3]{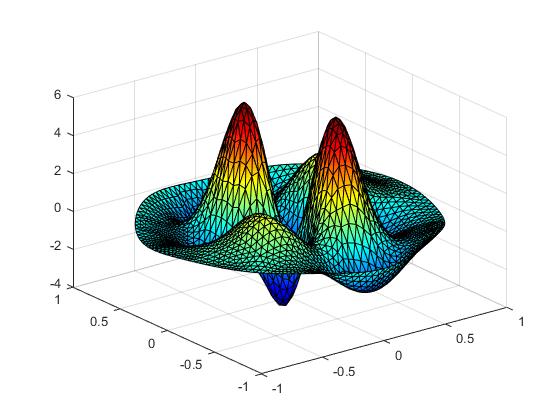} \\
(e) Reconstructed $f_2$ & (f) Reconstructed $f_2$
\end{tabular}
\caption{Reconstructions of $f=(f_1,f_2)$ from time-harmonic data at multi frequencies. Figures (b) and (e) are reconstructed from (\ref{Full-real}), whereas (c),(f) are from (\ref{Full-imag}).}
\label{TH-1}
\end{figure}

\begin{figure}[htbp]
\centering
\begin{tabular}{cc}
\includegraphics[scale=0.3]{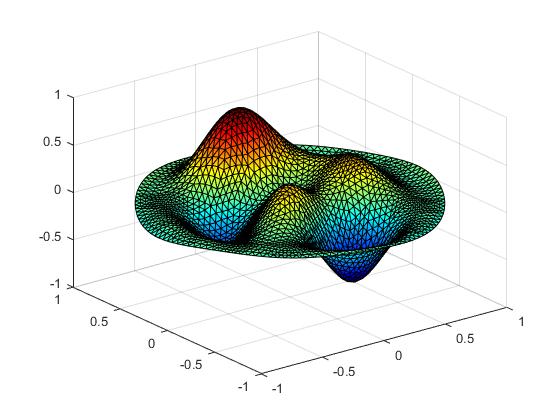} &
\includegraphics[scale=0.3]{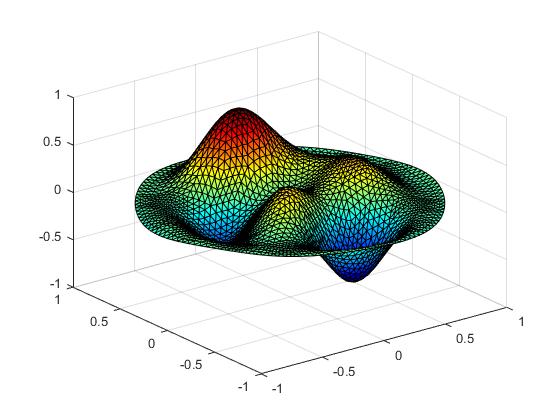} \\
(b) Reconstructed $f_p$ & (c) Reconstructed $f_p$ \\
\includegraphics[scale=0.3]{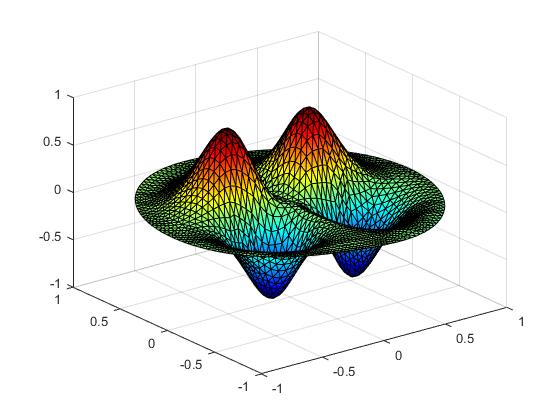} &
\includegraphics[scale=0.3]{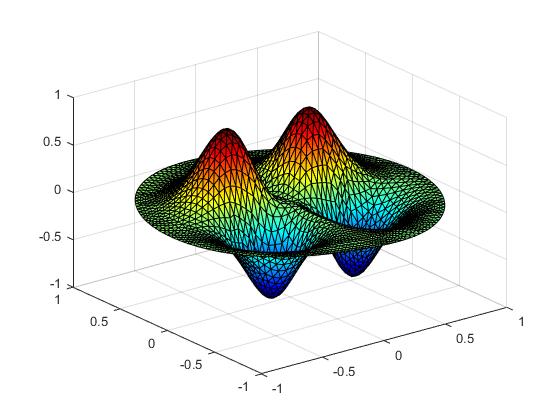} \\
(e) Reconstructed $f_s$ & (f) Reconstructed $f_s$
\end{tabular}
\caption{ Reconstructions of the compressional and shear components of $f$. Figures (b) and (e) are reconstructed from (\ref{ps-real}), whereas (c) and (f) are from (\ref{ps-imag}).}
\label{TH-2}
\end{figure}

In the time-dependent case, we first consider the numerical solution of the acoustic wave equation
\be
\label{eq:2d1}
&&\frac{1}{c_\alpha^2} \partial_{tt}\, u_\alpha(x,t)-\Delta u_\alpha(x,t)=g(t)\,f_\alpha(x),\;\mbox{in}\quad \R^2\times\R^+,\\
\label{eq:2d2}
&&u_\alpha|_{t=0}=\partial_t u_\alpha|_{t=0}=0\quad\mbox{in}\quad\R^2,\quad \alpha=p,s.
\en
To reduce the unbounded solution domain to a bounded computational domain, we use the local absorbing boundary condition
\ben
\partial_\nu u_\alpha+\frac{1}{c_\alpha}\partial_tu_\alpha+\frac{1}{2R}u_\alpha=0 \quad\mbox{on}\quad\partial B_{R}.
\enn
Then the solutions to the acoustic scattering problem (\ref{eq:2d1})-(\ref{eq:2d2}) are computed over $B_{R}$ by using interior penalty discontinuous Galerking method in space and  Newmark method in time. Consequently, the
 data $U(x,t)$ of the Lam\'e system are obtained through
\ben
U(x,t)=\frac{1}{\gamma_p}\mbox{grad}\,u_p(x,t) +\frac{1}{\gamma_s}\overrightarrow{\mbox{curl}}\,u_s(x,t).
\enn
In our numerical examples, we collect the scattering data $U(x,t)|_{\partial B_{R_1}}$ for $t\in[0,T]$ with $T=20>T_0+(R+R_1)/c_s=8$. In Figure \ref{comparison}, we compare the scattering data $\hat{u}(x,\omega)|_{\partial B_{R}}$
at frequencies $\omega=3$ and $\omega=10$ obtained by solving the time-harmonic Lam\'e system and that by applying Fourier transform  (denoted by $\hat{u}'(x,\omega)|_{\partial B_{R}}$) to the time-dependent data $U(x,t)|_{\partial B_{R}}$. It can be seen that the data set via Fourier transformation slightly differs from those time-harmonic data, possibly due to numerical errors in the Fourier transform and in the numerical scheme for solving time-dependent Lam\'e systems as well.
To Fourier transform the time domain data, we use fifteen equally spaced frequencies from 1 to 15. Numerical solutions for reconstructing $f$ and $f_\alpha, \alpha=p,s$ are presented in Figures \ref{TD-1} and \ref{TD-2}, respectively. We conclude from Figures \ref{TH-1}-\ref{TD-2} that satisfactory reconstructions are obtained through the proposed Landweber iterative algorithm.

\begin{figure}[htbp]
\centering
\begin{tabular}{cc}
\includegraphics[scale=0.3]{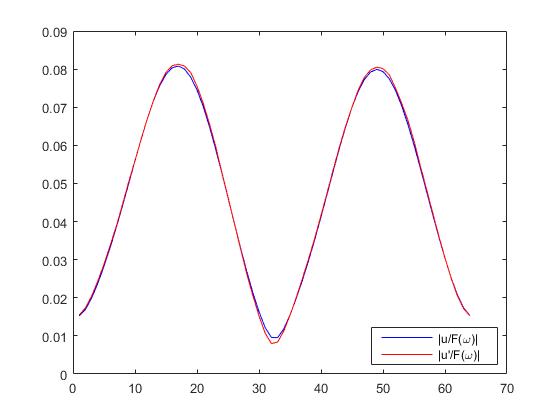} &
\includegraphics[scale=0.3]{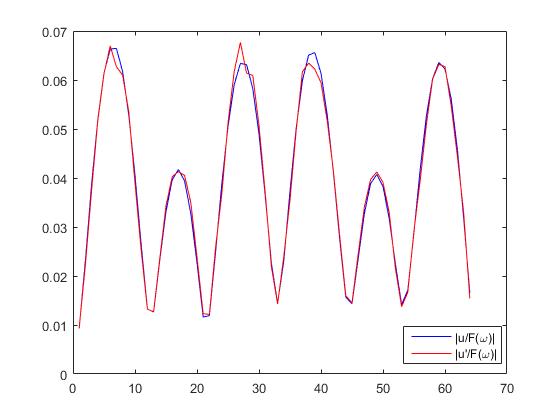} \\
(a) $\omega=3$ & (b) $\omega=10$  \\
\end{tabular}
\caption{Comparison of the scattering data $\hat{u}(x,\omega)/\hat{g}(\omega)$ and $\hat{u}'(x,\omega)/\hat{g}(\omega)$ at $\omega=3, 10$
obtained respectively by solving the time-harmonic Navier equation (blue) and by applying Fourier transform to the time-domain data (red).}
\label{comparison}
\end{figure}

\begin{figure}[htbp]
\centering
\begin{tabular}{cc}
\includegraphics[scale=0.3]{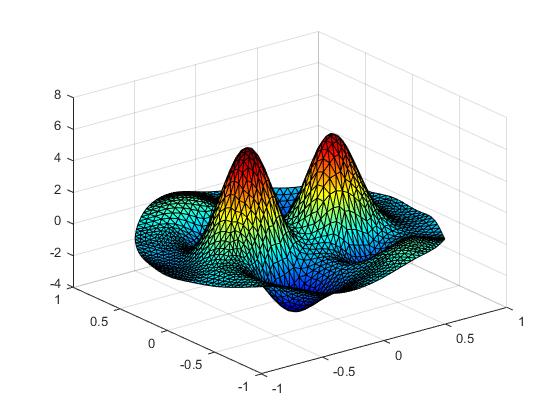} &
\includegraphics[scale=0.3]{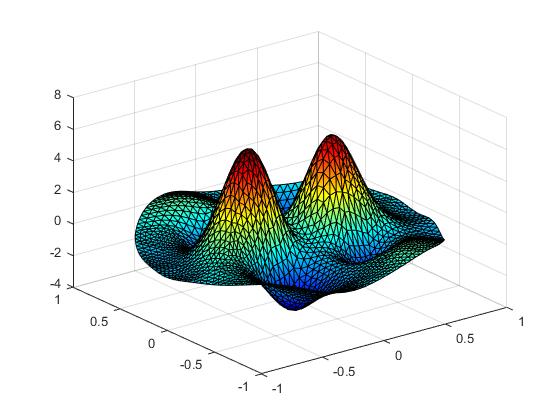} \\
(b) Reconstructed $f_1$ & (c) Reconstructed $f_1$ \\
\includegraphics[scale=0.3]{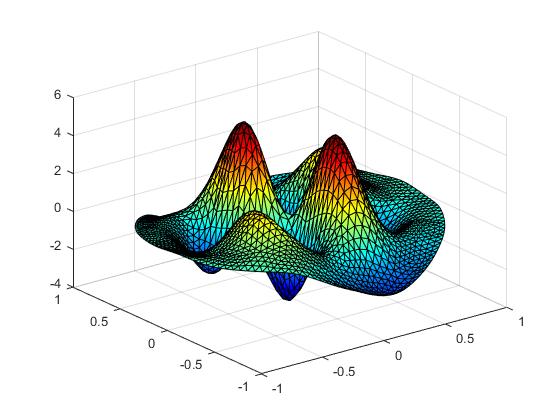} &
\includegraphics[scale=0.3]{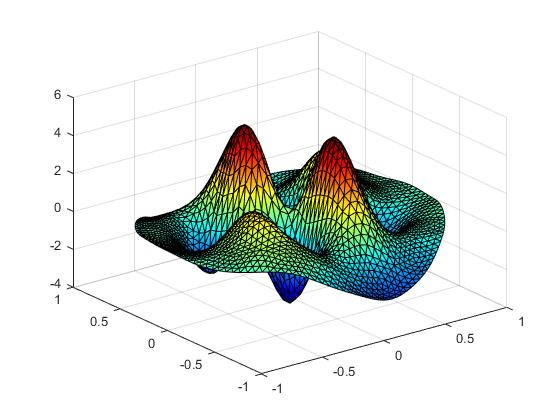} \\
(e) Reconstructed $f_2$ & (f) Reconstructed $f_2$
\end{tabular}
\caption{ Reconstructions of $f=(f_1,f_2)$ from Fourier-transformed time-domain scattering data. Figures (b) and (e) are reconstructed  from
(\ref{Full-real}), whereas (c) and (f) are from  (\ref{Full-imag}).}
\label{TD-1}
\end{figure}

\begin{figure}[htbp]
\centering
\begin{tabular}{cc}
\includegraphics[scale=0.3]{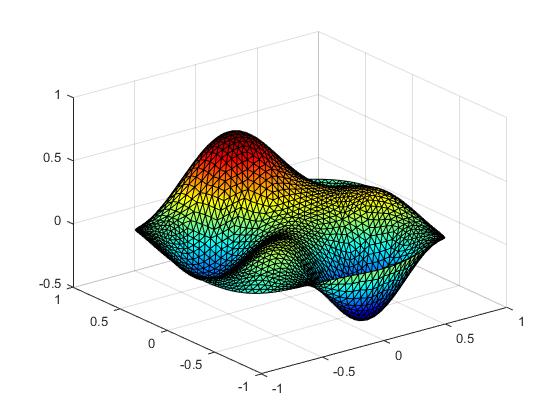} &
\includegraphics[scale=0.3]{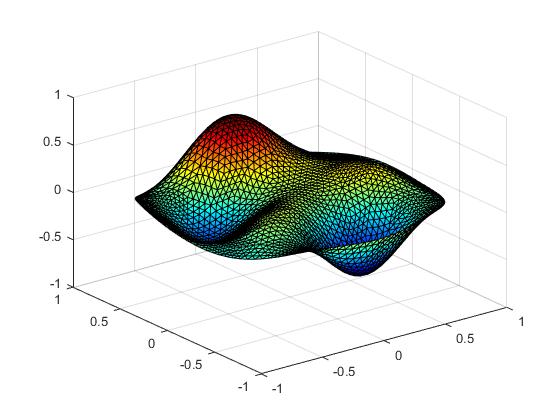} \\
(b) Reconstructed $f_p$ & (c) Reconstructed $f_p$ \\
\includegraphics[scale=0.3]{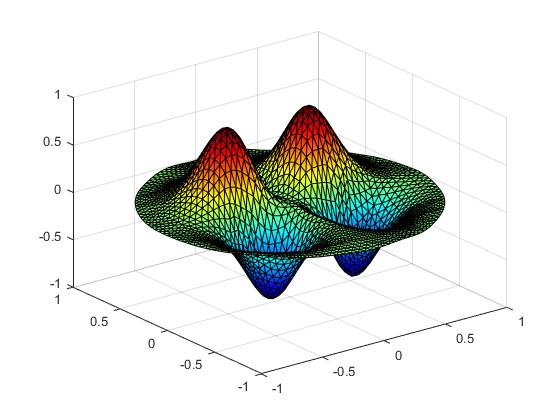} &
\includegraphics[scale=0.3]{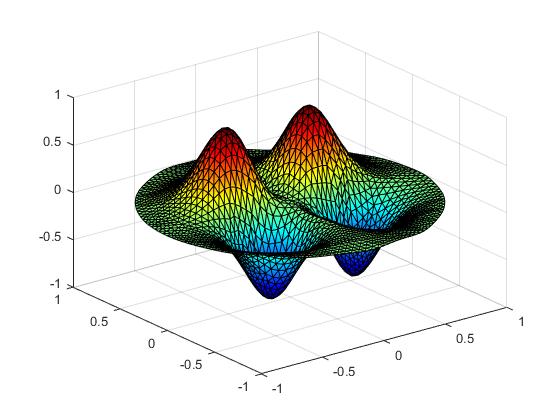} \\
(e) Reconstructed $f_s$ & (f) Reconstructed $f_s$
\end{tabular}
\caption{Reconstructions of the compressional and shear components of $f$ from Fourier-transformed time-domain scattering data. Figures (b) and (e) are reconstructed  from
(\ref{ps-real}), whereas (c) and (f) are from  (\ref{ps-imag}).}
\label{TD-2}
\end{figure}

\subsection{\textcolor{rot1}{Reconstruction of temporal functions}}\label{subsect:temporal}

\textcolor{rot2}{We consider the inverse problem of reconstructing $g$
from the wave fields $\{U(x,t): x\in \Gamma\subset \partial B_R, t\in(0, T)$ for some $T>0$ in three dimensions.}
For simplicity we choose the scalar spatial function to be the delta function, i.e., $f(x)=\delta(x)$. Then the function $W$ (see the proof of
Theorem \ref{TH:temporal}) takes the form
\ben
W(x,\omega)&=& \int_{\R^3}\hat{G}(x-y,\omega)f(y)\,dy\\
&=& \hat{G}(x,\omega)\\
&=& \frac{1}{\mu}\Phi_{k_s}(x) \textbf{I}+\frac{1}{\rho\omega^2}\, \grad_x\,\grad_x^\top\;\left[\Phi_{k_s}(x)-\Phi_{k_p}(x)\right],
\enn
where $\Phi_k(x)=e^{ik|x|}/(4\pi|x|)(k=k_p,k_s)$. Hence,  $f$ is indeed not a non-radiation source for all $\omega\in \R^+$.
In our example, we set the vector temporal function $g(t)$
to be
\ben
& g(t)=(g_1,g_2,g_3)^\top,\\
& g_1(t)=\begin{cases}
\cos(1.5\pi(t-t_1))\exp(-\pi(t-t_1)^2),& t\le T_1, \cr
0,& t>T_1,
\end{cases}\\
& g_2(t)=\begin{cases}
\sin(2\pi(t-t_2))\exp(-\pi(t-t_2)^2),& t\le T_2, \cr
0,& t>T_2,
\end{cases}\\
& g_3(t)=\begin{cases}
\sin(\pi(t-t_3))\exp(-\pi(t-t_3)^2),& t\le T_3, \cr
0,& t>T_3,
\end{cases}
\enn
where $T_1=5$, $T_2=4$, $T_3=3$, $t_1=2$, $t_2=3$ and $t_3=2$. The function pairs $(g_1,\hat{g}_1)$, $(g_2,\hat{g}_2)$ and
$(g_3,\hat{g}_3)$ are plotted in Figures \ref{pulse}, \ref{pulse2} and \ref{pulse3}, respectively.
Moreover, we set $g(t)=0$ for $t<0$.
\textcolor{rot2} {With the choice of $f$ and $g$, the forward time-domain scattering data can be expressed as $U=(u_1,u_2,u_3)$, where
\ben
u_i(x,t)&=& \sum_{j=1}^3\int_{0}^\infty\int_{\R^3} G_{i,j}(x-y, t-s) f(y)g_j(s)\,dx ds\\
&=& \sum_{j=1}^3\int_0^\infty G_{i,j}(x,t-s)g_j(s)\,ds\\
&=& \frac{1}{4\pi\rho|x|^3}\sum_{j=1}^3\left(\frac{x_ix_j}{c_p^2} g_j(t-|x|/c_p)+\frac{1}{c_s^2}(\delta_{ij}|x|^2- x_ix_j)g_j(t-|x|/c_s)
 \right) \\
&\quad&+\frac{1}{4\pi\rho|x|^3}\sum_{j=1}^3 \left(3x_ix_j-\delta_{ij}|x|^2\right) \int_{1/c_p}^{1/c_s} sg_j(t-s)\,ds.
\enn
Taking the Fourier transform gives the data $\hat{U}(x,\omega_i)$ in the Fourier domain.
The sampling frequencies are chosen as
\ben
\omega_j=1+(j-1)h,\quad h=19/49,\quad j=1,\cdots,K,\quad K=50.
\enn
Fixing $\omega_i\in \R^+$ ($i=1,2,\cdots, K$),
 we can always find $x_{0,i}\in \Gamma$ such that $W(x_{0,i},\omega_i)^{-1}$ exists and the value of
 the indicator
\ben
I_1(\omega_i)=[W(x_{0,i},\omega_i)]^{-1}\hat{U}(x_{0,i},\omega_i)
\enn
is identical to $\hat{g}(\omega_i)$. Taking the inverse Fourier transform of the indicator function $I_1(\omega)$ enables us to plot the function $t\rightarrow g_i(t)$ ($i=1,2, 3$).
In our tests we choose $x_{0,i}=(1,1,0)^\top$ uniformly in all $i=1,2,\cdots, K$.
Numerical reconstructions of $\hat{g}_i,i=1,2,3$ from the indicator $I_1$ are presented in Figure \ref{temporal1}.}

\textcolor{rot2} {One can readily observe that the choice of $x_{0,i}$ is not unique.
Our numerics show that $\mbox{Det}( W(x,\omega_i))$
does not vainish for almost all $x\in \partial B_R$. For $\omega_i\in \R^+$, we denote by $\{x_{j,i}: j=1,2,\cdots M\}$ a set of
points lying on $|x|=R$ such that $W(x_{j,i},\omega_i)$ is invertible for each $j$.
To make our inversion scheme computationally stable, we can calculate $I_1(\omega_i)$ using each $x_{j,i}$ ($j=1,2,\cdots,M$)
and then take the average  as the value of  $\hat{g}(\omega_i)$. Hence, we propose another
indicator function in the Fourier domain as following
\ben
I_2(\omega_i):=\frac{1}{M}\sum_{j=1}^M [W(x_{j,i},\omega_i)]^{-1}\hat{U}(x_{j,i},\omega_i),\quad i=1,2,\cdots,K,
\enn
where the time domain data $\{U(x_{j,i}, t): j=1,2,\cdots, M, i=1,2,\cdots, K\}$ are used. In our experiments,
we make use of the boundary data equivalently distributed on $|x|=R$ and set
\ben
x_{j,i}=x_j=(\cos((j-1)d\theta),1,\sin((j-1)h))^\top,\quad  h=2\pi/M, j=1,2,\cdots,M,
\enn uniformly in all $i=1,2,\cdots, K$. Numerics show that such kind of boundary data are adequate for the choice of $f$ and $g$.
Next we consider reconstructions from the noised data
\ben
U_\delta(x,t)=(1+\delta \epsilon(x,t))U(x,t)
\enn
where $\epsilon(x,t)$ is a function whose value is random between -1 and 1, and the noise level $\delta$ is set to be 30\%.
We present the reconstructions of $\hat{g}_j$ ($j=1,2,3$) based on the indicators $I_1$ and $I_2$
in Figures \ref{temporal2} and \ref{temporal3}, respectively.
Reconstructions from the inverse Fourier transform of $I_j$ (that is, the temporal function $g (t)$) are
illustrated in Figures  \ref{temporal4} and \ref{temporal5}, where
 the time-domain data with 30\% noise are again used.
 Comparing Figures \ref{temporal2}, \ref{temporal3},  \ref{temporal4} and \ref{temporal5},
 one may conclude that the inversion scheme using $I_2$ is indeed more computationally stable than $I_1$. }

\begin{figure}[htbp]
\centering
\begin{tabular}{cc}
\includegraphics[scale=0.3]{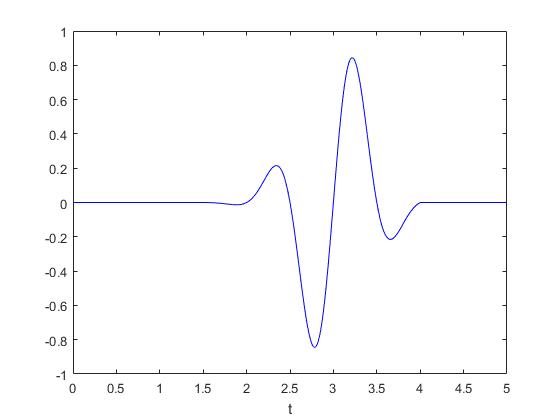} &
\includegraphics[scale=0.3]{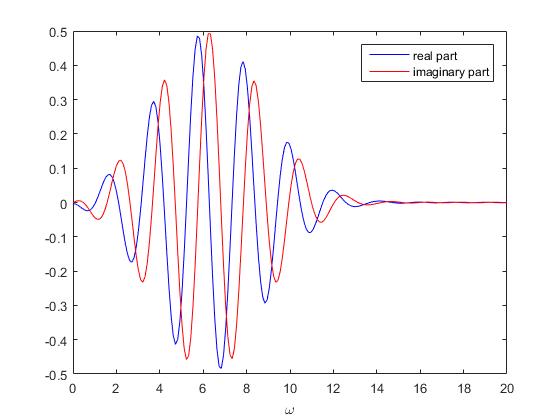} \\
(a) $g_2(t)$ & (b) $\hat{g}_2(\omega)$  \\
\end{tabular}
\caption{\textcolor{rot1}{The exact pulse function $g_2(t)$ and its Fourier transformation $\hat{g}_2(\omega)$.}}
\label{pulse2}
\end{figure}

\begin{figure}[htbp]
\centering
\begin{tabular}{cc}
\includegraphics[scale=0.3]{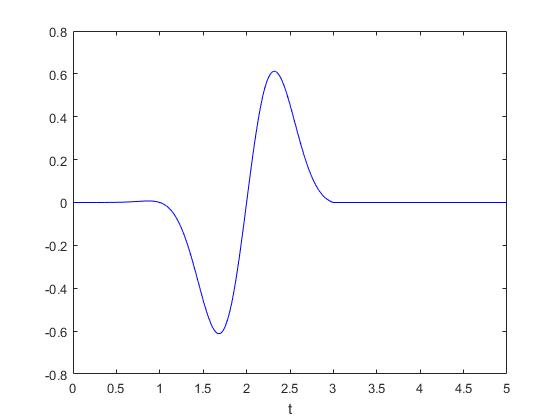} &
\includegraphics[scale=0.3]{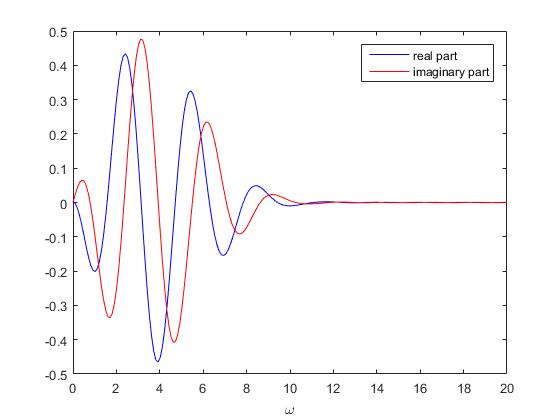} \\
(a) $g_2(t)$ & (b) $\hat{g}_2(\omega)$  \\
\end{tabular}
\caption{\textcolor{rot1}{The exact pulse function $g_3(t)$ and its Fourier transformation $\hat{g}_3(\omega)$.}}
\label{pulse3}
\end{figure}

\begin{figure}[htbp]
\centering
\begin{tabular}{ccc}
\includegraphics[scale=0.24]{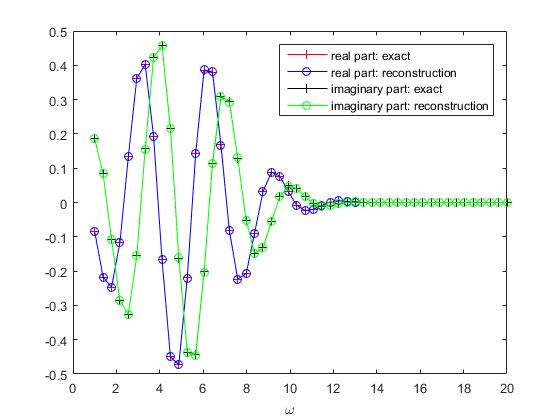} &
\includegraphics[scale=0.24]{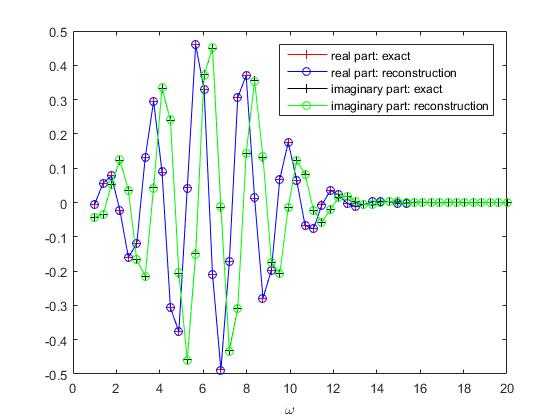} &
\includegraphics[scale=0.24]{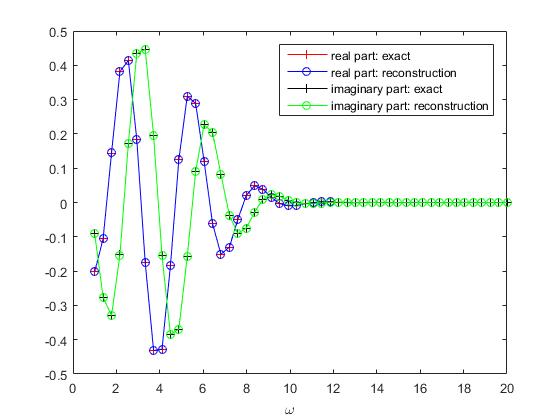} \\
(a) $\hat{g}_1$ & (a) $\hat{g}_2$ & (a) $\hat{g}_3$  \\
\end{tabular}
\caption{\textcolor{rot1}{Reconstruction of temporal functions from $I_1$ without noise.}}
\label{temporal1}
\end{figure}

\begin{figure}[htbp]
\centering
\begin{tabular}{ccc}
\includegraphics[scale=0.24]{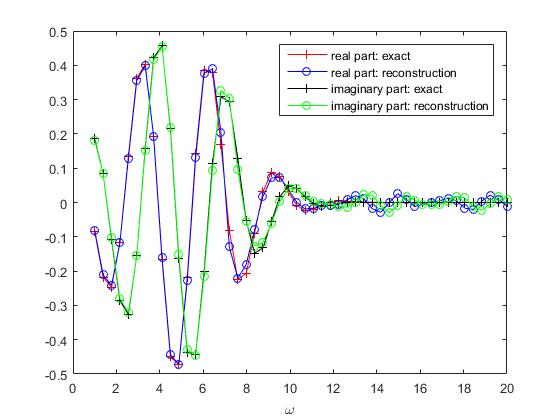} &
\includegraphics[scale=0.24]{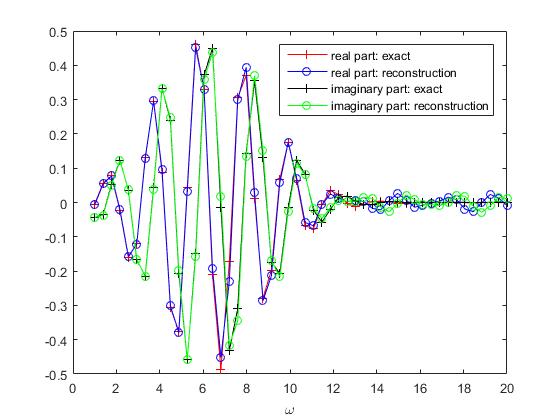} &
\includegraphics[scale=0.24]{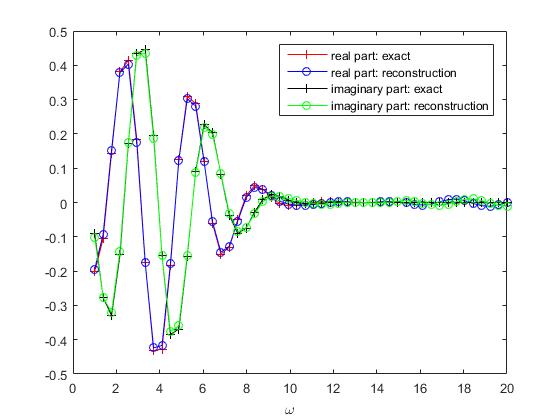} \\
(a) $\hat{g}_1$ & (a) $\hat{g}_2$ & (a) $\hat{g}_3$  \\
\end{tabular}
\caption{\textcolor{rot1}{Reconstruction of temporal functions from $I_1$ with 30\% noise.}}
\label{temporal2}
\end{figure}

\begin{figure}[htbp]
\centering
\begin{tabular}{ccc}
\includegraphics[scale=0.24]{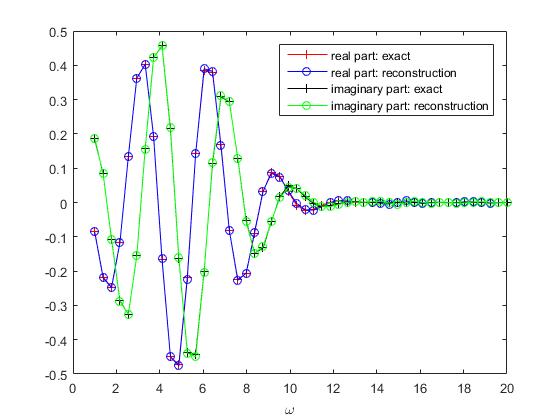} &
\includegraphics[scale=0.24]{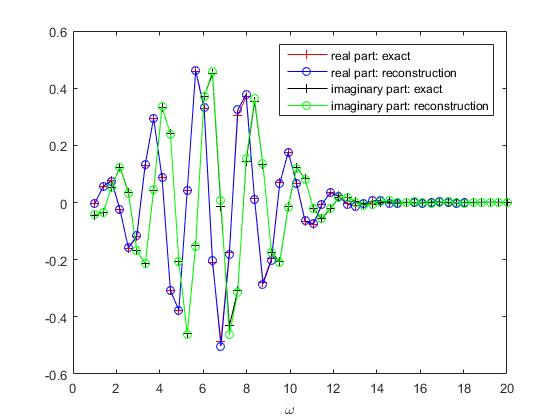} &
\includegraphics[scale=0.24]{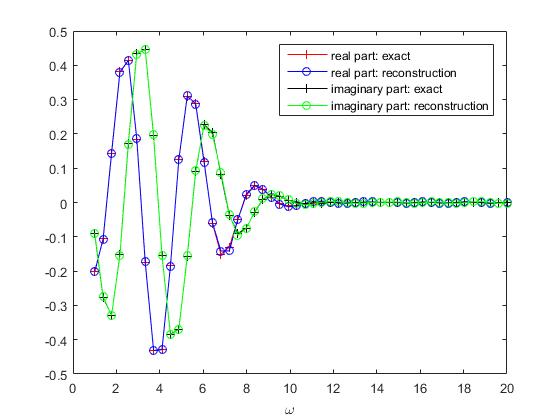} \\
(a) $\hat{g}_1$ & (a) $\hat{g}_2$ & (a) $\hat{g}_3$  \\
\end{tabular}
\caption{\textcolor{rot1}{Reconstruction of temporal functions from $I_2$ with 30\% noise.}}
\label{temporal3}
\end{figure}

\begin{figure}[htbp]
\centering
\begin{tabular}{ccc}
\includegraphics[scale=0.24]{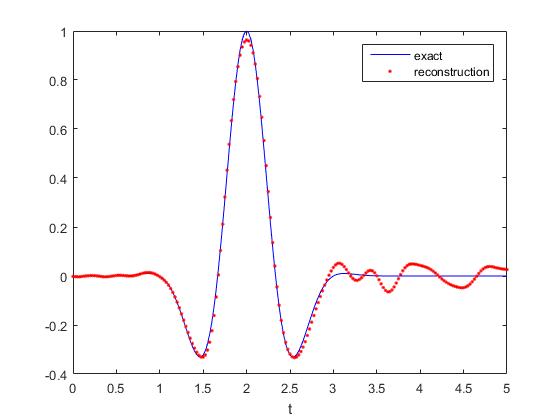} &
\includegraphics[scale=0.24]{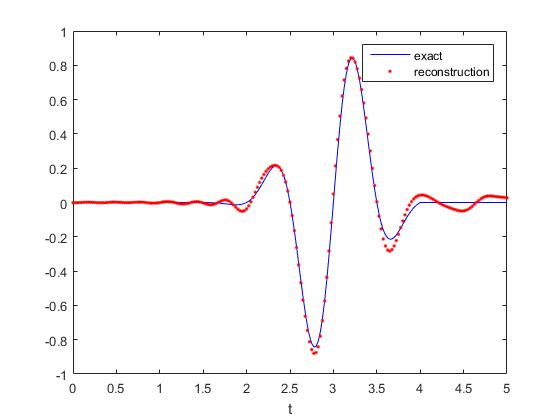} &
\includegraphics[scale=0.24]{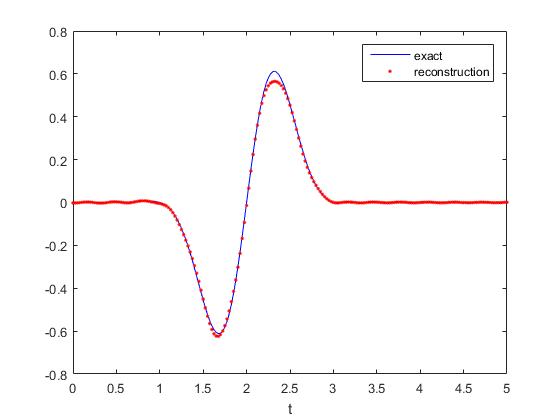} \\
(a) $g_1$ & (a) $g_2$ & (a) $g_3$  \\
\end{tabular}
\caption{\textcolor{rot1}{Reconstruction of temporal functions from $\tilde{I}_1$ with 30\% noise.}}
\label{temporal4}
\end{figure}

\begin{figure}[htbp]
\centering
\begin{tabular}{ccc}
\includegraphics[scale=0.24]{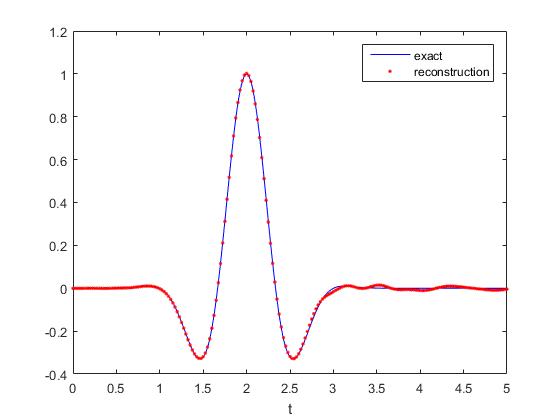} &
\includegraphics[scale=0.24]{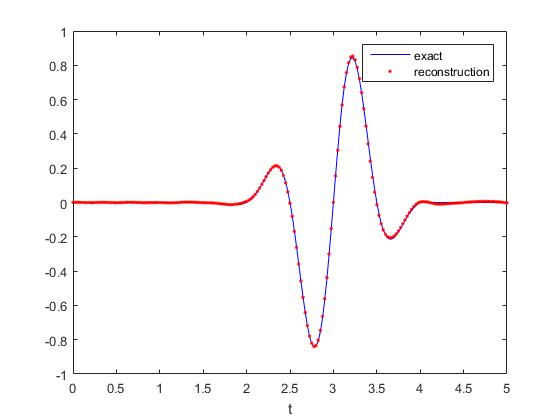} &
\includegraphics[scale=0.24]{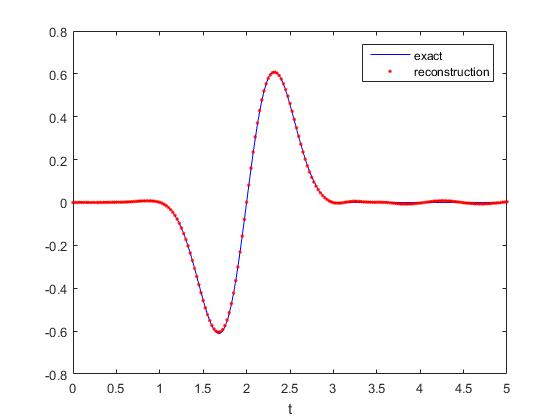} \\
(a) $g_1$ & (a) $g_2$ & (a) $g_3$  \\
\end{tabular}
\caption{\textcolor{rot1}{Reconstruction of temporal functions from $\tilde{I}_2$ with 30\% noise.}}
\label{temporal5}
\end{figure}

\section{Appendix}\label{Appendix}

\begin{lemma}
\label{HD-uniqueness}
Suppose that $S\in (L^2(\R^3))^3$ has a compact support in $B_R$ for some $R>0$, then the Helmholtz decomposition of $S$ is unique.
\end{lemma}
\begin{proof}
 Due to the Helmholtz decomposition, every $S\in (L^2(\R^3))^3$ admits a decomposition:
\ben
S=\nabla\,S_p+\nabla\times\,S_s,\quad \nabla\cdot S_s=0,
\enn
where
\ben
S_p\in H^1(B_{R}),\quad S_s\in H_{\curl}(B_{R}):=\{U: U\in L^2 (B_{R})^3, \curl U\in L^2 (B_{R})^3\}
\enn
 also have compact support in $B_R$. Suppose that $S$ admits another orthogonal decomposition $S=\nabla\,S_p'+\nabla\times\,S_s', \nabla\cdot S_s'=0$. Then we have
\be
\label{HD-uniqueness1}
\nabla\,(S_p-S_p')+\nabla\times\,(S_s-S_s')=0.
\en
Taking the divergence of both sides of (\ref{HD-uniqueness1}) gives $\Delta(S_p-S_p')=0$ in $B_R$, i.e., $S_p-S_p'$ is harmonic over $B_R$. Note that $S_p-S_p'=0$ on $\partial B_R$. Applying the maximum principle for harmonic functions yields $S_p=S_p'$ in $B_R$.
On the other hand, applying $\nabla\times$ to the both sides of (\ref{HD-uniqueness1}) we obtain
\ben
0=\nabla\times(\nabla\times\;(S_s-S_s'))=\nabla(\nabla\cdot\,(S_s-S_s')) -\Delta\,(S_s-S_s')= -\Delta\,(S_s-S_s').
\enn
Then the relation $S_s=S_s'$ in $B_R$ can be proved analogously. This completes the proof.
\end{proof}
In the following lemma, the notation $\textbf{I}_{n\times n}$ denotes the unit matrix in $\R^{n\times n}$ for $n\geq 2$.
\begin{lemma}\label{Lem:eigenvalue}
Let $\xi=(\xi_1,\cdots,\xi_n)^\top\in \R^{n\times 1}$ and $A(\xi)=\mu |\xi|^2 \textbf{I}_{n\times n}+(\lambda+\mu)\xi\otimes \xi\in \R^{n\times n}$. Then the eigenvalues $\tau_j$ ($j=1,2,\cdots,n$) of $A(\xi)$ are given by
\ben
\tau_1=(\lambda+2\mu)|\xi|^2,\quad \tau_2=\cdots=\tau_n=\mu\,|\xi|^2.
\enn
\end{lemma}
\begin{proof}
Set $\tilde{A}=A-\tau \textbf{I}_{n\times n}$. We may rewrite $\tilde{A}$ in the form $\tilde{A}=B+VV^\top$, where
\ben
B=(\mu |\xi|^2-\tau)\, \textbf{I}_{n\times n},\quad V=\sqrt{\lambda+\mu}\, \xi.
\enn
Straightforward calculations show that
\ben
\mbox{Det}(\tilde{A})&=&\mbox{Det}(B+VV^\top)\\
&=&(1+V^\top\, B^{-1}\,V)\,\mbox{Det}(B)\\
&=& \left(1+\frac{(\lambda+\mu)|\xi|^2}{\mu |\xi|^2-\tau}\right)\,(\mu |\xi|^2-\tau)^n\\
&=&[(\lambda+2\mu)|\xi|^2-\tau](\mu |\xi|^2-\tau)^{n-1},
\enn
which implies the eigenvalues of $A$.
\end{proof}

\begin{lemma}\label{Gronwall}(Grownwall-type inequality)
 Let $T>0$ and $u\in L^2(0,T)$ be  nonnegative  and fulfill, for almost every $t\in(0,T)$, the inequality
\begin{equation}\label{Gro1}u(t)\leq a(t)+\int_0^t b(s)\,u(s)\,ds,
\end{equation}
where $a\in L^2(0,T)$ and $b\in\mathcal C([0,T])$ are two nonnegative functions. Then, for almost every $t\in(0,T)$, we have

\begin{equation}\label{Gro2} u(t)\leq a(t)+\int_0^ta(s)b(s)e^{\int_s^tb(\tau)d\tau}ds.
\end{equation}
\end{lemma}
\begin{proof}
We consider $Y$ defined, for almost every $t\in(0,T)$, by
$$Y(t):=e^{-\int_0^tb(s)ds}\int_0^tb(s)u(s)ds$$
and we remark that $Y\in H^1(0,T)$ and satisfies $Y(0)=0$. Then, for almost every $t\in(0,T)$, we find
$$Y'(t)=b(t)u(t)e^{-\int_0^tb(s)ds}-b(t)e^{-\int_0^tb(s)ds}\int_0^tb(s)u(s)ds.$$
On the other hand, in view of \eqref{Gro1}, for almost every $t\in(0,T)$, we get
$$\int_0^tb(s)u(s)ds\geq u(t)-a(t)$$
and we deduce that
$$Y'(t)\leq a(t)b(t)e^{-\int_0^tb(s)ds}.$$
Integrating on both side of this inequality we get
$$\int_0^tY'(s)ds\leq \int_0^ta(s)b(s)e^{-\int_0^sb(\tau)d\tau}ds,\quad t\in(0,T).$$
On the other hand, since $Y\in H^1(0,T)$ and satisfies $Y(0)=0$, by density one can check that
$Y(t)=\int_0^tY'(s)ds$ which implies that
$$e^{-\int_0^tb(s)ds}\int_0^tb(s)u(s)ds\leq \int_0^ta(s)b(s)e^{-\int_0^sb(\tau)d\tau}ds$$
and by the same way, for almost every $t\in(0,T)$, the following inequality
$$\int_0^tb(s)u(s)ds\leq e^{\int_0^tb(s)ds}\left(\int_0^ta(s)b(s)e^{-\int_0^sb(\tau)d\tau}ds\right)=\int_0^ta(s)b(s)e^{\int_s^tb(\tau)d\tau}ds.$$
Finally, applying again \eqref{Gro1}, for almost every $t\in(0,T)$, we find
$$u(t)\leq a(t)+\int_0^tb(s)u(s)ds\leq a(t)+\int_0^ta(s)b(s)e^{\int_s^tb(\tau)d\tau}ds.$$
This proves \eqref{Gro2}.\end{proof}

\section*{Acknowledgement}
The work of G. Bao is supported in part by a NSFC Innovative Group Fun (No.11621101), an Integrated Project of the Major Research Plan of NSFC (No. 91630309), and an NSFC A3 Project (No. 11421110002). The work of G. Hu is supported by the NSFC grant (No. 11671028), NSAF grant (No. U1530401) and the 1000-Talent Program of Young Scientists in China. G. Hu and Y. Kian would like to thank Prof. M. Yamamoto for helpful discussions. The work of T. Yin is partially supported by the NSFC Grant (No. 11371385; No. 11501063).

\end{document}